\documentclass{amsart}
\usepackage{amsmath}
  \usepackage{paralist}
  \usepackage{graphics} 
  \usepackage{epsfig} 
\usepackage{graphicx}  \usepackage{epstopdf}
\usepackage{amssymb}
\usepackage{amsthm}
\usepackage{epstopdf}
\usepackage{verbatim}
\usepackage{mathrsfs}
\usepackage{mathtools}
\usepackage{pstricks}
 \usepackage[colorlinks=true]{hyperref}
\hypersetup{urlcolor=blue, citecolor=red}

\newtheorem{theorem}{Theorem}[section]

\newtheorem{lemma}[theorem]{Lemma}
\newtheorem{proposition}{Proposition}

\theoremstyle{definition}
\newtheorem{definition}[theorem]{Definition}
\newtheorem{remark}{Remark}

\title{Controllability under positivity constraints of multi-d wave equations}

\author[Dario Pighin and Enrique Zuazua]{}

\email{dario.pighin@uam.es}
\email{enrique.zuazua@uam.es}

\thanks{This research was supported by  the Advanced Grant DyCon (Dynamical Control) of the European Research Council Executive Agency  (ERC),  the MTM2014-52347 and MTM2017-92996 Grants of the MINECO (Spain) and  the ICON project of the French ANR-16-ACHN-0014.\\
	The authors thank Yubiao Zhang for his helpful revision of the paper.}

\begin{document}
	\maketitle
	
	\centerline{\scshape Dario Pighin$^*$}
	\medskip
	{\footnotesize
		\centerline{Departamento de Matem\'aticas, Universidad Aut\'onoma de Madrid}
		\centerline{28049 Madrid, Spain}
	} 

	\medskip
	
	\centerline{\scshape Enrique Zuazua}
	\medskip
	{\footnotesize
		\centerline{DeustoTech, Fundaci\'on Deusto}
		\centerline{Avda. Universidades,
			24, 48007, Bilbao, Basque Country, Spain}
	} 
	\medskip
	{\footnotesize
		\centerline{Departamento de Matem\'aticas, Universidad Aut\'onoma de Madrid}
		\centerline{28049 Madrid, Spain}
	} 
	\medskip
	{\footnotesize
		\centerline{Facultad de Ingenier\'ia, Universidad de Deusto}
		\centerline{Avda. Universidades,
			24, 48007, Bilbao, Basque Country, Spain}
	} 
	
	\bigskip
	

	\begin{abstract}
We consider both the internal  and boundary controllability problems  for  wave equations  under non-negativity constraints on the controls.		
First, we prove the steady state controllability property  with nonnegative controls for a general class of wave equations with time-independent coefficients. According to it, the system can be driven from a steady state generated by a strictly positive control to another, by means of nonnegative controls, and provided the time of control is long enough.  Secondly, under the added assumption of conservation and coercivity of the energy,  controllability is proved between states lying on two distinct trajectories. Our methods are described and developed in an abstract setting,  to be applicable to a wide variety of control systems.
		
	\end{abstract}
	\medskip
	\begin{center}
		\textit{Dedicated to Piermarco Cannarsa on the occasion of his 60th birthday}
	\end{center}
	\medskip
	
		\section{Introduction}
		\label{sec:1}
		
		This paper is devoted to the study of the controllability properties of the wave equation, under \textit{positivity} (or nonnegativity) constraints on the \textit{control}.
		
		We address both the case where the control acts in the \textit{interior} of the domain where waves evolve or on its  \textit{boundary}.
		
		This problem has been exhaustively considered in the  \textit{unconstrained} case but very little is known in the presence of constraints on the control, an issue of primary importance in applications, since whatever the applied context under consideration is, the available controls are always limited. For some of the basic literature on the unconstrained controllability of wave-like equations the reader is referred to:  \cite{bardos1992sharp}, \cite{burq1997condition}, \cite{cannarsa1999well}, \cite{cannarsa1999controllability}, \cite{dehman2009analysis}, \cite{ervedoza2010systematic}, \cite{lions1988exact}, \cite{HPR}, \cite{RCS}, \cite{zuazua1990exact}, \cite{zhang2004exact}.
		
		The developments in this paper are motivated by our earlier works on the constrained controllability of heat-like equations (\cite{HCC}, \cite{pighin2017controllability}). In that context, due to the well-known comparison principle for parabolic equations, control and state constraints are interlinked. In particular, for the heat equation, nonnegative controls imply that the solution is nonnegative too, when the initial configuration is nonnegative. Therefore, imposing non-negativity constraints on the control ensures that the state satisfies the non-negativity constraint too. 
		
		This is no longer true for wave-like equations in which the sign of the control does not determine that of solutions.
		However, as mentioned above, from a practical viewpoint, it is very natural to consider the problem of imposing control constraints. In this work, to fix ideas, we focus in the particular case of nonnegative controls.
		
		First we address the problem of steady state controllability in which one aims at controlling the solution from a steady configuration to another one. This problem was addressed in \cite{GCW}, in the absence of constraints on the controls for semilinear wave equations.  Our main contribution here is to control the system by preserving some constraints on the controls given a priori. And, as we shall see, when the initial and final steady states are associated to positive time-independent control functions, the constrained controllability can be guaranteed to hold if the time-horizon is long enough. 
		
		The proof is developed by  a step-wise procedure presented in \cite{pighin2017controllability} (which differs from the one in \cite{GCW} and \cite{HCC}), the so-called ``stair-case argument'', along an arc of steady-states linking the starting and final one. The proof consists on moving recursively from one steady state to the other by means of successive small amplitude controlled trajectories linking successive steady-states. This method and result are presented in a general semigroup setting and it can be successfully implemented for any control system for which controllability holds by means of $L\infty$ controls.
		
		The same recursive approach enables us to prove a state constrained result, under additional dissipativity assumptions. But the time needed for this to hold is even larger than before.
		
		The problem  of steady-state controllability is a particular instance of the more general trajectory control problem, in which, given two controlled  trajectories of the system, both obtained from nonnegative controls, and one state in each of them (possibly corresponding to two different time-instances) one aims at driving one state into the other one by means of nonnegative constrained controls.  This result can also be proved by a similar iterative procedure, but under the added assumption that the system is conservative and its energy coercive so that uncontrolled trajectories are globally bounded.
		
		These results hold for long enough control time horizons. The stepwise procedure we implement needs of a very large control time, much beyond the minimal control time for the control of the wave equation, that is determined by the finite velocity of propagation and the so-called Geometric Control Condition (GCC). It is then natural to introduce the minimal time of control under non-negativity constraints,  in both situations above. 
		
		There is plenty to be done to understand how these constrained minimal times depends on the data to be controlled. Employing d'Alembert's formula for the one dimensional wave equation, we  compute both of them for constant steady states, showing that they  coincide with the unconstrained one. In that case we also show that the property of constrained controllability holds in the minimal time too.

		Controllability under constraints has already been studied for finite-dimensional models and heat-like equations (see \cite{HCC} and \cite{pighin2017controllability}). 
		In both cases it was also proved  that controllability by nonnegative controls fails if time is too short, when the initial datum differs from the final target. This fact exhibits a big difference with respect to  the unconstrained control problem for these systems, where controllability holds in arbitrary small time in both cases. In the wave-like context addressed in this paper the waiting phenomenon, according to which there is a minimal control time for the constrained problem, is  less surprising. But, simultaneously, on the other hand, in some sense, the fact that constraints can be imposed on controls and state seems more striking too.
		
		In \cite{gugat2011optimal}, authors analysed controllability of the one dimensional wave equation, under the more classical bilateral constraints on the control. Our work is, as far as we know, the first one considering unilateral constraints for wave-like equations.

		\subsection{Internal control}
		\label{subsec:1.2}
		Let $\Omega$ be a connected bounded open set of $\mathbb{R}^n$, $n \ge 1$, with $C^{\infty}$ boundary, and let $\omega$ and $\omega_0$ be subdomains of $\Omega$ such that $\overline{\omega_0}\subset \omega$.
		
		Let $\chi\in C^{\infty}(\mathbb{R}^n)$ be a smooth function supported in $\omega$ such that $\mbox{Range}(\chi)\subseteq [0,1]$, $\chi\hspace{-0.1 cm}\restriction_{\omega_0}\equiv 1$.
		
		We assume further that all derivatives of $\chi$ vanish on the boundary of $\Omega$. We will discuss this assumption in subsection \ref{subsec:3.3}.
		
		We consider the wave equation controlled from the interior
		\begin{equation}\label{wave_internal_1}
		\begin{cases}
		y_{tt}-\Delta y+cy=u\chi\hspace{0.6 cm} & \mbox{in} \hspace{0.10 cm}(0,T)\times \Omega\\
		y=0  & \mbox{on}\hspace{0.10 cm} (0,T)\times \partial \Omega\\
		y(0,x)=y^0_0(x), \ y_t(0,x)=y^1_0(x)  & \mbox{in} \hspace{0.10 cm}\Omega\\
		\end{cases}
		\end{equation}
		where $y=y(t,x)$ is the state, while $u=u(t,x)$ is the control whose action is localized on $\omega$ by means of  multiplication with the smooth cut-off function $\chi$. The coefficient $c=c(x)$ is $C^{\infty}$ smooth in $\overline{\Omega}$.
		
		It is well known in the literature (e.g. \cite[section 7.2]{PDE}) that, for any initial datum $(y_0^0,y_0^1)\in H^1_0(\Omega)\times L^2(\Omega)$ and for any control $u\in L^2((0,T)\times\omega)$, the above problem admits an unique solution $(y,y_t)\in C^0([0,T];H^1_0(\Omega)\times L^2(\Omega))$, with $y_{tt}\in L^2(0,T;H^{-1}(\Omega))$.
		
		We assume the \textit{Geometric Control Condition} on $(\Omega,\omega_0,T^*)$, which basically asserts that all  bicharacteristic rays enter in the subdomain $\omega_0$ in time smaller than $T^*$. This geometric condition is actually equivalent to the property of (unconstrained) \textit{controllability} of the system (see \cite{bardos1992sharp} and \cite{burq1997condition}).
		
		\subsubsection{Steady state controllability}
		
		The purpose of our first result is to show that, in time large, we can drive \eqref{wave_internal_1} from one steady state to another by a \textit{nonnegative} control, assuming the uniform \textit{positivity} of the control defining the steady states.
		
		More precisely, a steady state is a solution to
		\begin{equation}\label{wave_internal_steady_1}
		\begin{cases}
		-\Delta \overline{y}+c\overline{y}=\overline{u}\chi\hspace{0.6 cm} & \mbox{in} \hspace{0.10 cm} \Omega\\
		\overline{y}=0 & \mbox{on}\hspace{0.10 cm}  \partial \Omega,\\
		\end{cases}
		\end{equation}
		where $\overline{u}\in L^2(\omega)$ and $\overline{y}\in H^2(\Omega) \cap H^1_0(\Omega)$. Note that,  as a consequence of Fredholm Alternative (see \cite[Theorem 5.11 page 84]{EPG}), the existence and uniqueness of the solution of this elliptic problem  can be guaranteed whenever zero is not an eigenvalue of $-\Delta+cI:H^1_0(\Omega)\longrightarrow H^{-1}(\Omega)$.

		The following result holds:
		
		\begin{theorem}[Controllability between steady states]\label{th_2}
			Take $\overline{y}_0$ and $\overline{y}_1$ in $H^2(\Omega) \cap H^1_0(\Omega)$ steady states associated to $L^2$-controls $\overline{u}^1$ and $\overline{u}^2$, respectively. Assume further that there exists $\sigma>0$ such that
			\begin{equation}\label{th_2_uniform_positiviveness_control}
			\overline{u}^i\geq \sigma,	\quad\mbox{a.e. in}\ \omega.
			\end{equation}
			Then, if $T$ is large enough, there exists 
			$u\in L^2((0,T)\times\omega)$, a control such that
			\begin{itemize}
				\item the unique solution $(y,y_t)$ to the problem \eqref{wave_internal_1} with initial datum $(\overline{y}_0,0)$ and control $u$ verifies $(y(T,\cdot),y_t(T,\cdot))=(\overline{y}_1,0)$;
				\item $
				u\geq 0$ a.e. on $(0,T)\times \omega.$
			\end{itemize}
		\end{theorem}
		
		Theorem \ref{th_2} is proved in subsection \ref{subsec:3.1}. Inspired by \cite{GCW}, we implement a recursive ``stair-case'' argument to keep the control  in a narrow tubular neighborhood of the segment connecting the controls defining the initial and final data. This will guarantee the actual positivity of the control obtained.
		
		
		\subsubsection{Controllability between trajectories}
		
		The purpose of this section is to extend the above result, under the additional assumption $c(x)>-\lambda_1$, where $\lambda_1$ is the first eigenvalue of the Dirichlet Laplacian in $\Omega$. This guarantees that the energy of the system defines a norm
		\begin{equation*}
		\|(y^0,y^1)\|_{E}^2=\int_{\Omega} \left[\|\nabla y^0\|^2 +c{\left(y^0\right)}^2\right] dx+\int_{\Omega} (y^1)^2 dx
		\end{equation*}
		on $H^1_0(\Omega)\times L^2(\Omega)$. Thus, by conservation of the energy, uncontrolled solutions are uniformly bounded for all $t$.

		
		\begin{figure}[htp]
			\begin{center}
				\includegraphics[width=8cm]{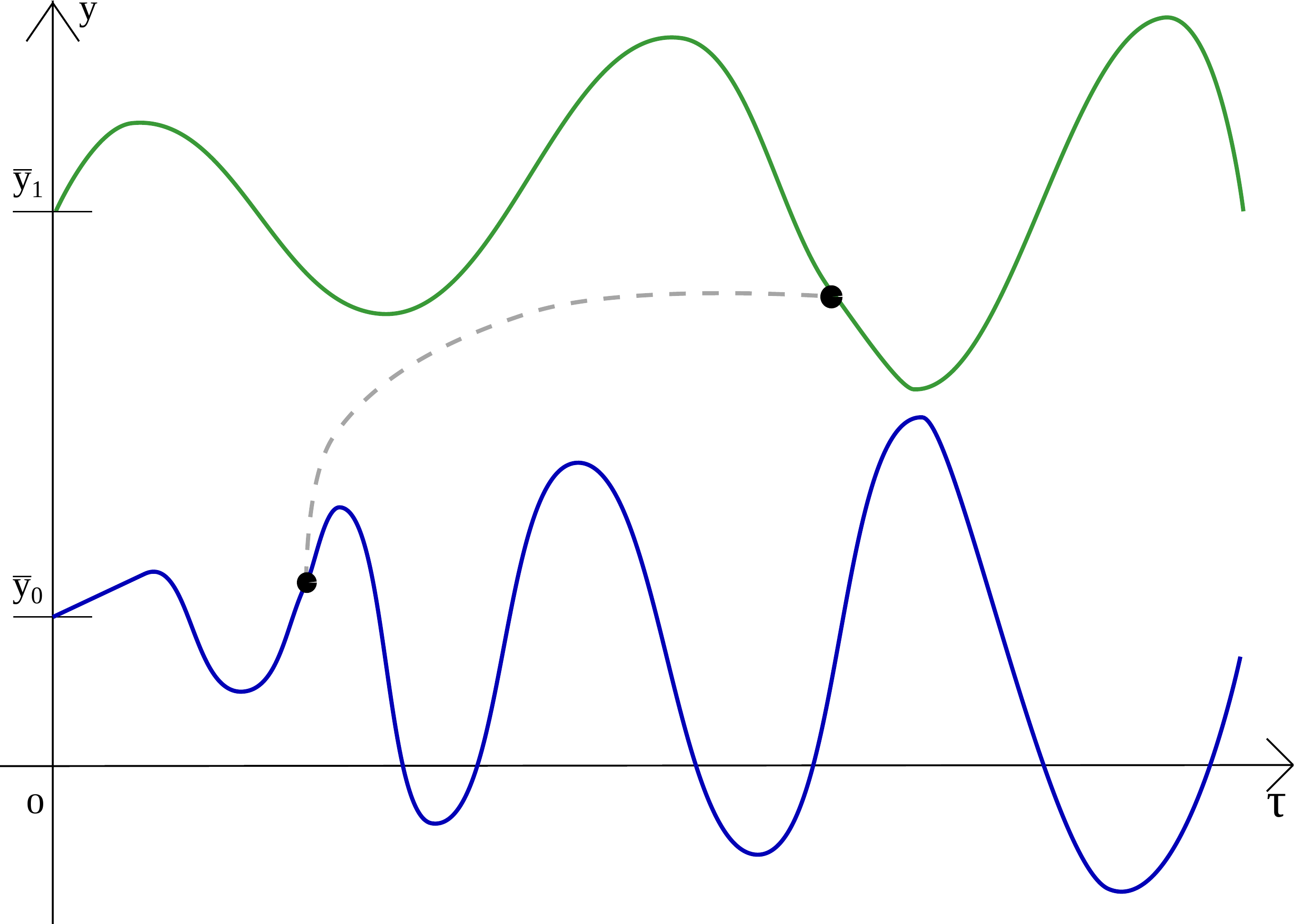}\\
				\caption{Controllability between data lying on trajectories.}\label{INDAM 2018 intro}
			\end{center}
		\end{figure}
		
		We assume that both, the initial datum $(y_0^0,y_0^1)$ and the final target $(y_1^0,y_1^1)$, belong  to controlled trajectories (see figure \ref{INDAM 2018 intro}) \begin{equation}\label{initial_final_elements_trajectories}
		(y_i^0,y_i^1)\in \left\{ (\overline{y}_i(\tau,\cdot),(\overline{y}_i)_t(\tau,\cdot) \ | \ \tau\in \mathbb{R}\right\},
		\end{equation}
		where $(\overline{y}_i,(\overline{y}_i)_t)$ solve \eqref{wave_internal_1} with \textit{nonnegative} controls. We suppose that these trajectories are smooth enough, namely
		$$(\overline{y}_i,(\overline{y}_i)_t)\in C^{s(n)}(\mathbb{R};H^1_0(\Omega)\times L^2(\Omega)),$$ with $s(n)=\lfloor{n/2}\rfloor+1$. Hereafter, we denote by $(\overline{y}_0,(\overline{y}_0)_t)$ the initial trajectory, while $(\overline{y}_1,(\overline{y}_1)_t)$ stands for the target one.
		
		Note that the regularity is assumed only in time and not in space. This allows to consider weak steady-state solutions.
		
		We can in particular choose as final target the null state $(y_1^0,y_1^1)=(0,0)$. It is important to highlight that this is something specific to the wave equation. In the parabolic case (see \cite{HCC} and \cite{pighin2017controllability}), this was prevented by the comparison principle, since the zero target cannot be reached in finite time with non-negative controls. But, for the wave equation, the maximum principle does not hold and this obstruction does not apply.
		
		The following result holds
		
		\begin{theorem}[Controllability between trajectories]\label{th_3}
			Suppose $c(x)>-\lambda_1$, for any $x\in\overline{\Omega}$. Let $(\overline{y}_i,(\overline{y}_i)_t)\in C^{s(n)}(\mathbb{R};H^1_0(\Omega)\times L^2(\Omega))$  be solutions to \eqref{wave_internal_1} associated to controls $\overline{u}^i\geq 0$ a.e. in $(0,T)\times \omega$, $i=0, 1$. Take $(y_0^0,y_0^1)=(\overline{y}_0(\tau_0,\cdot),(\overline{y}_0)_t(\tau_0,\cdot))$ and $(y_1^0,y_1^1)=(\overline{y}_1(\tau_1,\cdot),(\overline{y}_1)_t(\tau_1,\cdot))$ for arbitrary values of $\tau_0$ and $\tau_1$. Then, in time $T>0$  large enough, there exists a control $u\in L^{2}((0,T)\times \omega)$ such that
			\begin{itemize}
				\item the unique solution $(y,y_t)$ to \eqref{wave_internal_1} with initial datum $(y_0^0,y_0^1)$ verifies the end condition $(y(T,\cdot),y_t(T,\cdot))=(y_1^0,y_1^1)$;
				\item $u\geq 0$ a.e. in $(0,T)\times \omega$.
			\end{itemize}
		\end{theorem}
		
		\begin{remark}
			This result  is more general than Theorem \ref{th_2} for two reasons
			\begin{enumerate}
				\item it enables us to link more general data, with nonzero \textit{velocity}, and not only steady states;
				\item the control defining the initial and target trajectories is assumed to be only \textit{nonnegative}. This assumption is weaker than the \textit{uniform positivity} one required in Theorem \ref{th_2}.
			\end{enumerate}
			
			On the other hand, the present result requires the condition $c(x)>-\lambda_1$ on the potential $c=c(x)$.
		\end{remark}
		
		We give the proof of Theorem \ref{th_3} in subsection \ref{subsec:3.2}.
		
		\subsection{Boundary control}
		\label{subsec:1.3}

		Let $\Omega$ be a connected bounded open set of $\mathbb{R}^n$, $n \ge 1$, with $C^{\infty}$ boundary, and let $\Gamma_0$ and $\Gamma$ be open subsets of $\partial \Omega$ such that $\overline{\Gamma_0}\subset \Gamma$.
		
		Let $\chi\in C^{\infty}(\partial \Omega)$ be a smooth function such that $\mbox{Range}(\chi)\subseteq [0,1]$, $\mbox{supp}(\chi)\subset \Gamma$ and $\chi\hspace{-0.1 cm}\restriction_{\Gamma_0}\equiv 1$.
		
		We now consider the wave equation controlled on the \textit{boundary}
		\begin{equation}\label{wave_boundary_1}
		\begin{cases}
		y_{tt}-\Delta y+cy=0\hspace{0.6 cm} & \mbox{in} \hspace{0.10 cm}(0,T)\times \Omega\\
		y=\chi u  & \mbox{on}\hspace{0.10 cm} (0,T)\times \partial \Omega\\
		y(0,x)=y^0_0(x), \ y_t(0,x)=y^1_0(x)  & \mbox{in} \hspace{0.10 cm}\Omega\\
		\end{cases}
		\end{equation}
		where $y=y(t,x)$ is the state, while $u=u(t,x)$ is the boundary control localized on $\Gamma$ by the cut-off function $\chi$. As before, the space-dependent coefficient $c$ is supposed to be $C^{\infty}$ regular in $\overline{\Omega}$.
		
		By transposition (see \cite{lions1988exact})
		, one can realize that for any initial datum $(y_0^0,y_0^1)\in L^2(\Omega)\times H^{-1}(\Omega)$ and control $u\in L^2((0,T)\times\Gamma)$, the above problem admits an unique solution $(y,y_t)\in C^0([0,T];L^2(\Omega)\times H^{-1}(\Omega))$.

		We assume the \textit{Geometric Control Condition} on $(\Omega,\Gamma_0,T^*)$ which asserts that all generalized bicharacteristics touch the sub-boundary $\Gamma_0$ at a non diffractive point in time smaller than $T^*$. By now, it is well known in the literature that this geometric condition is equivalent to (unconstrained) controllability (see \cite{bardos1992sharp} and \cite{burq1997condition}).
		
		
		\subsubsection{Steady state controllability}
		
		As in the context of internal control, our first goal is to show that, in time large, we can drive \eqref{wave_internal_1} from one steady state to another, assuming the uniform \textit{positivity} of the controls defining these steady states.
		
		In the present setting a steady state is a time independent solution to \eqref{wave_boundary_1}, namely a solution to
		\begin{equation}\label{steady_wave_boundary_1}
		\begin{cases}
		-\Delta \overline{y}+c\overline{y}=0\hspace{0.6 cm} & \mbox{in} \hspace{0.10 cm}\Omega\\
		\overline{y}=\chi \overline{u}  & \mbox{on}\hspace{0.10 cm}\partial \Omega.\\
		\end{cases}
		\end{equation}
		In the present setting, $\overline{u}\in L^2(\partial\Omega)$ and  $\overline{y}\in L^2(\Omega)$ solves the above problem in the sense of transposition (see \cite[chapter II, section 4.2]{LIO} and \cite{LM1}).
		
		As in the context of  internal control, if $0$ is not an eigenvalue of
		$-\Delta+cI:H^1_0(\Omega)\longrightarrow H^{-1}(\Omega)$, for any boundary control $\overline{u}\in L^2(\partial\Omega)$, there exists a unique $\overline{y}\in L^2(\Omega)$ solution to \eqref{steady_wave_boundary_1} with boundary control $\overline{u}$. This can be proved combining Fredholm Alternative (see \cite[Theorem 5.11 page 84]{EPG}) and transposition techniques \cite[Theorem 4.1 page 73]{LIO}.

		We prove the following result
		\begin{theorem}[Steady state controllability]\label{th_1}
			Let $\overline{y}_i$  be steady states defined by controls $\overline{u}^i$, $i=0, 1$, so that 
			\begin{equation}\label{th_1_uniform_positiviveness_control}
			\overline{u}^i\geq \sigma,	\quad\mbox{on}\ \Gamma,
			\end{equation}
			with $\sigma>0$.
			
			Then, if $T$ is large enough, there exists 
			$u\in L^{2}([0,T]\times\Gamma)$, a control such that
			\begin{itemize}
				\item the unique solution $(y,y_t)$ to \eqref{wave_boundary_1} with initial datum $(\overline{y}_0,0)$ and control $u$ verifies $(y(T,\cdot),y_t(T,\cdot))=(\overline{y}_1,0)$;
				\item $
				u\geq 0$ on $(0,T)\times \Gamma.$
			\end{itemize}
		\end{theorem}
		
		The proof of the above result can be found in subsection \ref{subsec:4.1}. The structure of the proof resembles the one of Theorem \ref{th_2}, with some technical differences due to the different nature of the control.
		
		\subsubsection{Controllability between trajectories}
		
		As in the internal control case, we suppose $c(x)>-\lambda_1$, where $\lambda_1$ is the first eigenvalue of the Dirichlet Laplacian in $\Omega$. Then, the generator of the free dynamics is skew-adjoint (see \cite[Proposition 3.7.6]{OCO}), thus generating an unitary group of operators $\left\{\mathbb{T}_{t}\right\}_{t\in\mathbb{R}}$ on $L^2(\Omega)\times H^{-1}(\Omega)$.

		Both the initial datum and final target  $(y_i^0,y_i^1)$  belong to a smooth trajectory, namely
		\begin{equation}\label{initial_final_elements_trajectories_boundary}
		(y_i^0,y_i^1)\in \left\{ (\overline{y}_i(\tau,\cdot),(\overline{y}_i)_t(\tau,\cdot)) \ | \ \tau\in \mathbb{R}\right\}.
		\end{equation}
		We assume the \textit{nonnegativity} of the controls $\overline{u}^i$ defining $(\overline{y}_i,(\overline{y}_i)_t)$, for $i=0,1$. Hereafter, in the context of boundary control, we take trajectories of class\\
		$C^{s(n)}(\mathbb{R};L^2(\Omega)\times H^{-1}(\Omega))$, with $s(n)=\lfloor{n/2}\rfloor+1$. We set  $(\overline{y}_0,(\overline{y}_0)_t)$ to be the initial trajectory and  $(\overline{y}_1,(\overline{y}_1)_t)$ be the target one.
		
		Note that, with respect to Theorem \ref{th_1}, we have relaxed the assumptions on the sign of the controls $\overline{u}^i$. Now, they are required to be only \textit{nonnegative} and not uniformly strictly positive.
		
		\begin{theorem}[Controllability between trajectories]\label{th_4}
			Assume $c(x)>-\lambda_1$, for any $x\in\overline{\Omega}$. Let  $(\overline{y}_i,(\overline{y}_i)_t)$  be solutions to \eqref{wave_boundary_1} with non-negative controls $\overline{u}^i$ respectively. Suppose the trajectories $(\overline{y}_i,(\overline{y}_i)_t)\in  C^{s(n)}([0,T];L^2(\Omega)\times H^{-1}(\Omega))$. Pick $(y_0^0,y_0^1)=(\overline{y}_0(\tau_0,\cdot),(\overline{y}_0)_t(\tau_0,\cdot))$ and $(y_1^0,y_1^1)=(\overline{y}_1(\tau_1,\cdot),(\overline{y}_1)_t(\tau_1,\cdot))$. Then, in time large, we can find a control $u\in L^{2}((0,T)\times \Gamma)$ such that
			\begin{itemize}
				\item the solution $(y,y_t)$ to \eqref{wave_boundary_1} with initial datum $(y_0^0,y_0^1)$ fulfills the final condition $(y(T,\cdot),y_t(T,\cdot))=(y_1^0,y_1^1)$;
				\item $u\geq 0$ a.e. in $(0,T)\times \Gamma$.
			\end{itemize}
		\end{theorem}
		
		
		The above Theorem is proved in subsection \ref{subsec:4.2}. Furthermore, in section \ref{sec:5}, we show how Theorem \ref{th_4} applies in the one dimensional case, providing further information about the minimal time to control and the possibility of controlling the system  in the minimal time.
		
		\subsubsection{State constraints}
		
		We impose now constraints both on the control and on the state, namely both the \textit{control} and the \textit{state} are required to be nonnegative.
		
		In the parabolic case (see \cite{HCC} and \cite{pighin2017controllability}) one can employ the comparison principle to get a state constrained result from a control constrained one. But, now, as we have explained before, the comparison principle is not valid in general for the wave equation.
		And we cannot rely on comparison to deduce our state constrained result from the control constrained one.
		
		We shall rather apply the ``stair-case argument'' developed to prove steady state controllability, paying attention to the added need of preserving state constraints as well.
		
		
		Let $\lambda_1$ be the first eigenvalue of the Dirichlet Laplacian. We assume $c> -\lambda_1$ in $\overline{\Omega}$. We also suppose that $\chi\equiv 1$, meaning that the control acts on the whole boundary. We take as initial and final data two steady states $y_0^0$ and $y_1^0$ associated to controls $\overline{u}^i\geq \sigma>0$. Our proof relies on the application of the maximum principle to \eqref{steady_wave_boundary_1}. This ensures that the states $\overline{y}_i\geq \sigma$ once we know $\overline{~u}^i\geq \sigma$. For this reason, we need $c> -\lambda_1$ and $\chi\equiv 1$.
		
		Our strategy is the following
		\begin{itemize}
			\item employ the ``stair-case argument'' used to prove steady state controllability, to keep the control in a narrow tubular neighborhood of the segment connecting $\overline{u}^0$ and $\overline{u}^1$. This can be done by taking the time of control  large enough. Since $\overline{u}^i\geq \sigma>0$,  this guarantees the positivity of the control;
			\item by the continuous dependence of the solution on the data, the controlled trajectory remains also in a narrow neighborhood of the convex combination joining initial and final data. On the other hand, by the maximum principle for the steady problem \eqref{steady_wave_boundary_1}, we have that $y_i^0\geq \sigma$ in $\Omega$, for $i=0, 1$. In this way the state  $y$ can be assured to remain nonnegative.
		\end{itemize}
		
		\begin{theorem}\label{th_5}
			We assume $c(x)> -\lambda_1$ for any $x\in \overline{\Omega}$ and $\chi\equiv 1$. Let $y_0^0$ and $y_1^0$ be solutions to the steady problem
			\begin{equation}\label{steady_wave_boundary_2}
			\begin{cases}
			-\Delta y+cy=0\hspace{0.6 cm} & \mbox{in} \hspace{0.10 cm}\Omega\\
			y=\overline{u}^i,  & \mbox{on}\hspace{0.10 cm}\partial \Omega\\
			\end{cases}
			\end{equation}
			where $\overline{u}^i\geq \sigma$ a.e. on $\partial\Omega$, with $\sigma>0$. We assume $y_i^0\in H^{s(n)}(\Omega)$. Then, there exists $\overline{T}>0$ such that for any $T>\overline{T}$ there exists a control $u\in L^{\infty}((0,T)\times \partial\Omega)$ such that
			\begin{itemize}
				\item the unique solution $(y,y_t)$ to \eqref{wave_boundary_1} with initial datum $(y_0^0,0)$ and control $u$ is such that $(y(T,\cdot),y_t(T,\cdot))=(y_1^0,0)$;
				\item $u\geq 0$ a.e. on $(0,T)\times \partial\Omega$;
				\item $y\geq 0$ a.e. in $(0,T)\times \Omega$.
			\end{itemize}
		\end{theorem}
		
		The proof of the above Theorem can be found in subsection \ref{subsec:4.3}.
		
		Note that the time needed to control the system keeping both the control and the state \textit{nonnegative} is greater (or equal) than the corresponding one with no constraints on the state.
		
		\subsection{Orientation}
		
		The rest of the paper is organized as follows:
		\begin{itemize}
			\item Section \ref{sec:2}: Abstract results;
			\item Section \ref{sec:3}: Internal Control: Proof of Theorem \ref{th_2} and Theorem \ref{th_3};
			\item Section \ref{sec:4}: Boundary control: Proof of Theorem \ref{th_1}, Theorem \ref{th_4} and Theorem \ref{th_5};
			\item Section \ref{sec:5}: The one dimensional case;
			\item Section \ref{sec:6}: Conclusion and open problems;
			\item Appendix.
		\end{itemize}
		
		\section{Abstract results}
		\label{sec:2}
		
		The goal of this section is to provide some results on constrained controllability for some abstract control systems. We apply these results in the context of internal control and boundary control of the wave equation (see Section \ref{sec:1}).
		
		We begin introducing the abstract control system. Let $H$ and $U$ be two Hilbert spaces
		endowed with norms $\|\cdot\|_{H}$ and $\|\cdot\|_{U}$ respectively. $H$ is called the {state} space and $U$ the {control} space. Let $A:D(A)\subset H\longrightarrow H$ be a generator of a ${C}_0$-semigroup $(\mathbb{T}_t)_{t\in\mathbb{R}^+}$, with $\mathbb{R}^+= [0,+\infty)$. The domain of the generator $D(A)$ is endowed with the graph norm $\|x\|_{D(A)}^2=\|x\|_{H}^2+\|Ax\|_{H}^2$. We define $H_{-1}$ as the completion of $H$ with respect to the norm $\|\cdot\|_{-1}=\|(\beta I-A)^{-1}(\cdot)\|_{H}$, with real $\beta$ such that
		$(\beta I-A)$ is invertible from $H$ to $H$ with continuous inverse. Adapting the techniques of \cite[Proposition 2.10.2]{OCO}, one can check that the definition of $H_{-1}$ is actually independent of the choice of $\beta$. By applying the techniques of \cite[Proposition 2.10.3]{OCO}, we deduce that $A$ admits a unique bounded extension $A$ from $H$ to $H_{-1}$. For simplicity, we still denote by $A$ the extension. Hereafter, we write $\mathscr{L}(E,F)$ for the space of all bounded linear operators from a Banach space $E$ to another Banach space $F$. 
		
		Our control system is governed by:
		\begin{equation}\label{state_equation_abstract}
		\begin{cases}
		\frac{d}{dt}y(t)=Ay(t)+Bu(t),& t\in(0,\infty),\\
		y(0)=y_0,\\
		\end{cases}
		\end{equation}
		where $y_0\in H$, $u\in L^2_{loc}([0,+\infty),U)$ is a control function and the {control} operator $B\in \mathscr{L}(U,H_{-1})$ satisfies the  {admissibility} condition in the following definition (see \cite[Definition 4.2.1]{OCO}).
		\begin{definition}\label{def_admiss}
			The control operator $B\in \mathscr{L}(U,H_{-1})$ is said to be admissible if for all  $\tau>0$ we have $\mbox{Range}(\Phi_{\tau})\subset H$, where
			$\Phi_{\tau}:L^2((0,+\infty);U)\to H_{-1}$ is defined by:
			\begin{equation*}
			\Phi_{\tau}u=\int_{0}^{\tau}\mathbb{T}_{\tau-r}Bu(r)dr.
			\end{equation*}
		\end{definition}
		From now on, we will always assume the control operator to be admissible. One can check that for any $y_0\in H$ and $u\in L^2_{loc}((0,+\infty);U)$ there exists a unique mild solution $y\in C^0([0,+\infty),H)$ to \eqref{state_equation_abstract} (see, for instance, \cite[Proposition 4.2.5]{OCO}). We denote by $y(\cdot;y_0,u)$ the unique solution to \eqref{state_equation_abstract} with initial datum $y_0$ and control $u$.
		
		Now, we introduce the following constrained controllability problem
		\begin{itemize}
			\item[]\textit{Let $\mathscr{U}_{\mbox{\tiny{ad}}}$ be a nonempty subset of $U$. Find a subset $E$ of $H$ so that for each\\
				$y_0, \,y_1 \in E$, there exists $T>0$ and a control $u\in L^{\infty}(0,T;U)$ with $u(t)\in \mathscr{U}_{\mbox{\tiny{ad}}}$ for a.e. $t\in (0,T)$, so that $y(T;y_0,u)=y_1$}.
		\end{itemize}
		We address this controllability problem in the next two subsections, under different assumptions on $\mathscr{U}_{\mbox{\tiny{ad}}}$ and $(A,B)$. In Subsection \ref{subsec:2.1}, we study the above controllability problem, where the initial and final data are steady states, i.e. solutions to the steady equation:
		\begin{equation}\label{steady_state_equation_abstract}
		Ay+Bu=0 \quad\mbox{for some}\quad u\in U.
		\end{equation}
		In Subsection \ref{subsec:2.2}, we take initial and final data on two different trajectories of \eqref{state_equation_abstract}.
		
		To study the above problem, we need two ingredients, which play a key role in the proofs of Subsection \ref{subsec:2.1} and Subsection \ref{subsec:2.2}. First, we introduce the notion of smooth controllability. Before introducing this concept, 
		%
		%
		%
		we fix  $s\in \mathbb{N}$ and a Hilbert space $V$ so that
		\begin{equation}\label{s and V}
		\quad V \hookrightarrow U,
		\end{equation}
		where $ \hookrightarrow$  denotes the continuous embedding. Note that all throughout the remainder of the section, $s$ and $V$ remain fixed. 
		
		The concept of smooth controllability is given in the following definition. The notation $y(\cdot;y_0,u)$ stands for the solution of the abstract controlled equation \eqref{state_equation_abstract} with control $u$ and initial data $y_0$.
		
		\begin{definition}\label{contr_2_abstract}
			The control system \eqref{state_equation_abstract} is said to be {smoothly controllable} in time $T_0>0$ if for any $y_0\in D(A^s)$, there exists a control function $v\in L^{\infty}((0,T_0);V)$ such that
			\begin{equation*}
			y(T_0;y_0,v)=0
			\end{equation*}
			and
			\begin{equation}\label{estimate_contr_2_abstract}
			\|v\|_{L^{\infty}((0,T_0);V)}\leq C\|y_0\|_{D(A^s)},
			\end{equation}
			the constant $C$ being independent of $y_0$.
		\end{definition}

		\begin{remark}
			(i) In other words, the system is smoothly controllable in time $T_0$ if for each (regular) initial datum $y_0\in D(A^s)$, there exists a $L^{\infty}$-control $u$ with
			values in the regular space $V$ steering our control system to rest at time $T_0$.
			
			(ii) The smooth controllability in time $T_0$ of system \eqref{state_equation_abstract} is a consequence of the following observability inequality:  there exists a constant $C> 0$ such that for any $z\in D(A^*)$
			\begin{equation*}
			\|\mathbb{T}_{T_0}^*z\|_{D(A^s)^*}\leq C\int_0^{T_0} \|i^{*}B^*\mathbb{T}_{T_0-t}^{*}z\|_{V^*}dt,
			\end{equation*}
			where $D(A^s)^*$ is the dual of $D(A^s)$ and $i:V \hookrightarrow U$ is the inclusion. This inequality, that can often be proved out of classical observability inequalities employing the regularizing properties of the system,  provides a way to prove the smooth controllability for system \eqref{state_equation_abstract}.  This occurs for parabolic problem enjoying smoothing properties.
			
			(iii)  Besides, for some systems $(A,B)$, even if they do not enjoy smoothing properties,  there is an alternative way to prove the aforementioned smooth controllability property exploiting the ellipticity properties of the control operator (see \cite{ervedoza2010systematic}).
			
			Under suitable assumptions, the wave system is smoothly controllable (see Lemma \ref{Lemma_smooth_contr_wave_int} and Lemma \ref{lemma_smooth_controllability_wave_boundary}).

		\end{remark}

		
		
		The second ingredient is following lemma, which concerns the regularity of the inhomogeneous problem.
		
		\begin{lemma}\label{lemma_semigroup_4}
			Fix $k\in\mathbb{N}$ and take $f\in H^{k}((0,T);H)$ such that
			\begin{equation}\label{lemma_semigroup_4_hypo_1}
			\begin{cases}
			\frac{d^j}{dt^j}f(0)=0,\quad&\forall \ j\in \left\{0,\dots,k \right\}\\
			f(t)=0,\quad&\mbox{a.e.} \ t\in (\tau,T),\\
			\end{cases}
			\end{equation}
			with $0<\tau<T$. Consider $y$ solution to the problem
			\begin{equation}\label{state_equation_2}
			\begin{cases}
			\frac{d}{dt}y=Ay+f& t\in (0,T)\\
			y(0)=0.\\
			\end{cases}
			\end{equation}
			Then, $y\in \cap_{j=0}^k C^j([\tau,T];D(A^{k-j}))$ and
			\begin{equation*}
			\sum_{j=0}^k\|y\|_{C^j([\tau,T];D(A^{k-j}))}\leq C\|f\|_{H^k((0,T);H)},
			\end{equation*}
			the constant $C$ depending only on $k$.
		\end{lemma}
		
		\begin{remark}
			Note that the maximal regularity of the solution is only assured for $t \ge \tau$, after the right hand side  term $f$ vanishes.
		\end{remark}
		
		The proof of this Lemma is given in an Appendix at the end of this paper.

		\subsection{Steady state controllability}
		\label{subsec:2.1}
		
		In this subsection, we study the constrained controllability for some steady states. Recall $s$ and $V$ are given by \eqref{s and V}. Before introducing our main result, we suppose:
		
		\hspace{0.03 cm}$(H_1)$ the system \eqref{state_equation_abstract} is smoothly controllable in time $T_0$ for some $T_0>0$.
		
		\hspace{0.03 cm}$(H_2)$ $\mathscr{U}_{\mbox{\tiny{ad}}}$ is a closed and convex cone with vertex at $0$ and $\mbox{int}^V(\mathscr{U}_{\mbox{\tiny{ad}}}\cap V)\neq \varnothing$,
		
		\hspace{0.80 cm}where $\mbox{int}^V$ denotes the interior
		set in the topology of $V$.\\
		Furthermore, we define the following subset
		\begin{equation}\label{def_W}
		\mathscr{W}=\mbox{int}^V(\mathscr{U}_{\mbox{\tiny{ad}}}\cap V)+\mathscr{U}_{\mbox{\tiny{ad}}}.
		\end{equation}
		(Note that, since $\mathscr{U}_{\mbox{\tiny{ad}}}$ is a convex cone, then $\mathscr{W}\subset \mathscr{U}_{\mbox{\tiny{ad}}}$.) The main result of this subsection is the following. The solution to \eqref{state_equation_abstract} with initial datum $y_0$ and control $u$ is denoted by $y(\cdot;y_0,u)$.
		
		\begin{theorem}[Steady state controllability]\label{th_6}
			Assume $(H_1)$ and $(H_2)$ hold. Let\\
			$\left\{(y_i,\overline{u}^i)\right\}_{i=0}^1\subset H\times \mathscr{W}$ satisfying
			\begin{equation*}
			Ay_i+B\overline{u}^i=0,\qquad i=0,1.
			\end{equation*}
			Then there exists $T>T_0$ and $u\in L^2(0,T;U)$ such that
			\begin{itemize}
				\item $
				u(t)\in \mathscr{U}_{\mbox{\tiny{ad}}}$ a.e. in $(0,T)$;
				\item $y(T;y_0,u)=y_1$.
			\end{itemize}
		\end{theorem}
		
		\begin{remark}\label{remark_subsection_2.1_1}
			%
			As we shall see, in the application to the wave equation with positivity constraints:
			\begin{itemize}
				\item for \textit{internal control}, $U=L^2(\omega)$ and $V=H^{s(n)}(\omega)$, with $s=s(n)=\lfloor{n/2}\rfloor+1$;
				\item for \textit{boundary control}, $U=L^2(\Gamma)$ and $V=H^{s(n)-\frac12}(\Gamma)$, where $s(n)=\lfloor{n/2}\rfloor+1$.
			\end{itemize}
			$\mathscr{U}_{\mbox{\tiny{ad}}}$ is the set of nonnegative controls in $U$. In both cases, $\mathscr{W}$ is nonempty and contains controls $u$ in $L^2(\omega)$ (resp. $L^2(\Gamma)$) such that $u\geq \sigma$, for some $\sigma>0$. For this to happen, it is essential that $H^{s(n)}(\omega)\hookrightarrow C^0(\overline{\omega})$ (resp. $H^{s(n)-\frac12}(\Gamma)\hookrightarrow C^0(\overline{\Gamma})$). This is guaranteed by our special choice of $s=s(n)$. Furthermore, in these special cases:
			\begin{equation*}
			\overline{\mathscr{W}}^{U}=\mathscr{U}_{\mbox{\tiny{ad}}},
			\end{equation*}
			where $\overline{\mathscr{W}}^{U}$ is the closure of $\mathscr{W}$ in the space $U$.
		\end{remark}

		
		
		
		
		


		In the remainder of the present subsection we prove Theorem \ref{th_6}. The following Lemma is essential for the proof of Theorem \ref{th_6}. Fix $\rho\in C^{\infty}(\mathbb{R})$ such that \begin{equation}\label{lemma_10_hypo}
		\mbox{Range}(\rho)\subseteq [0,1],\quad\rho\equiv 1\quad \mbox{over} \ (-\infty,0] \quad\mbox{and}\quad \mbox{supp}(\rho)\subset \subset (-\infty,1/2).
		\end{equation}
		
		\begin{lemma}\label{lemma_10} 
			Assume that the system \eqref{state_equation_abstract} is smoothly controllable in time $T_0$, for some $T_0>0$. Let $(\eta_0,
			\overline{v}^0)\in H\times U$ be a steady state, i.e. solution to \eqref{steady_state_equation_abstract} with control $\overline{v}^0$. Then, there exists $w\in L^{\infty}((1,T_0+1);V)$ such that the control
			\begin{equation}\label{control_noreg_smooth}v(t)=
			\begin{cases}
			\rho(t)\overline{v}^0 \quad &\mbox{in} \ (0,1)\\
			w\quad &\mbox{in} \ (1,T_0+1)
			\end{cases}
			\end{equation}
			drives \eqref{state_equation_abstract} from $\eta_0$ to $0$ in time $T_0+1$.
			Furthermore,
			\begin{equation}\label{control_norm_estimate_1}
			\|w\|_{L^{\infty}((1,T_0+1);V)}\leq C\|\eta_0\|_{H}.
			\end{equation}
		\end{lemma}
		The proof of the above Lemma can be found in the Appendix.

		We prove now Theorem \ref{th_6}, by developing a ``stair-case argument'' (see Figure \ref{Stepwise procedure}).
		\begin{figure}[htp]
			\begin{center}
				\includegraphics[width=10cm]{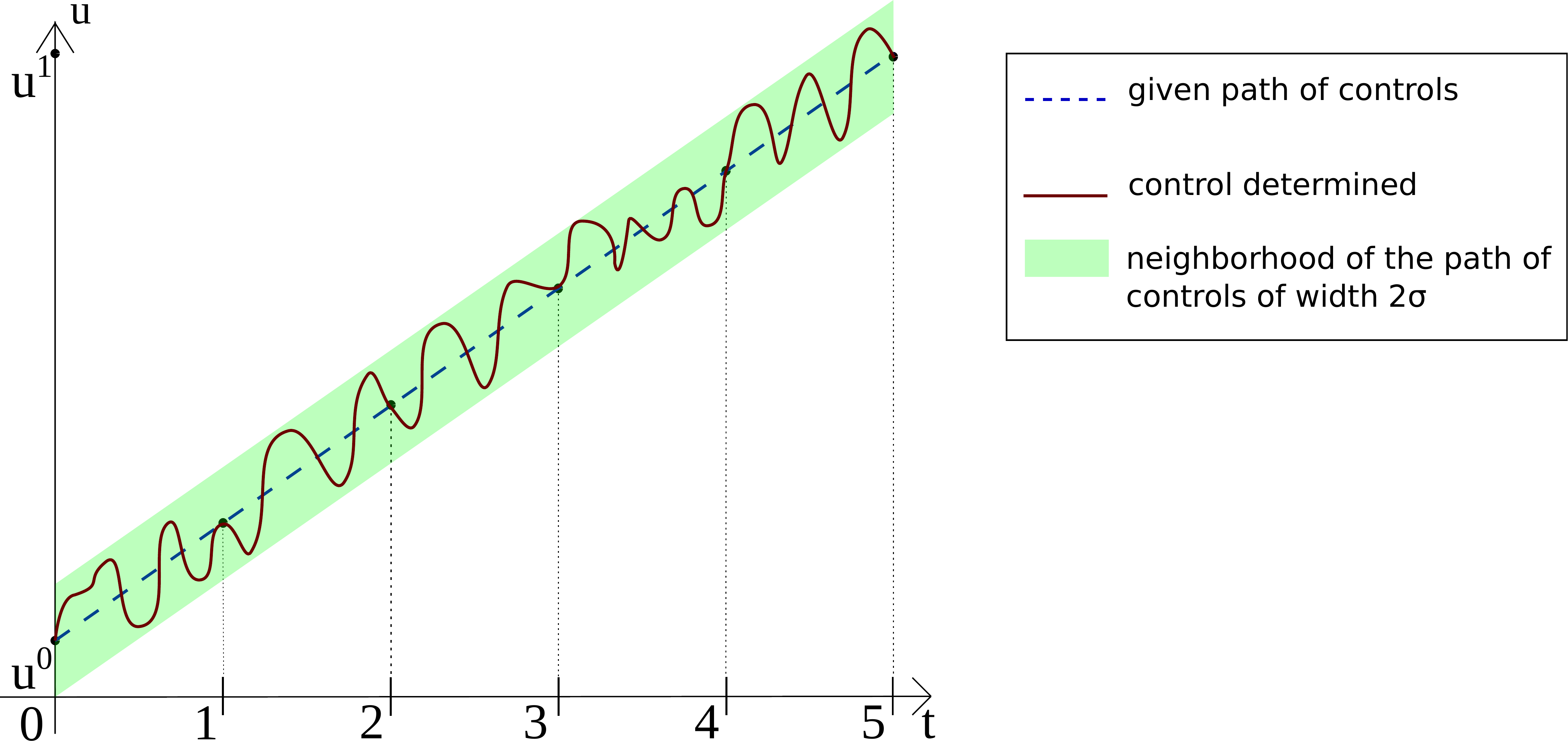}\\
				\caption{Stepwise procedure}\label{Stepwise procedure}
			\end{center}
		\end{figure}
		
		\begin{proof}[Proof of Theorem \ref{th_6}]

			Let $\left\{(y_i,\overline{u}^i)\right\}_{i=0}^1$ satisfy
			\begin{equation}\label{th_6_eq_1}
			Ay_i+B\overline{u}^i=0 \quad \forall \ i\in \left\{0,1\right\}.
			\end{equation}
			By the definition of $\mathscr{W}$, there exists $\left\{(q^i,z^i)\right\}_{i=0}^1\subset \mbox{int}^V(\mathscr{U}_{\mbox{\tiny{ad}}}\cap V)\times \mathscr{U}_{\mbox{\tiny{ad}}}$ such that
			\begin{equation}\label{th_6_eq_2}
			\overline{u}^i=q^i+z^i \quad i=0,1.
			\end{equation}
			Define the segment joining $y_0$ and $y_1$
			\begin{equation*}
			\gamma(s)=(1-s) y_0+s y_1\quad \forall \ s\in [0,1].
			\end{equation*}
			For each $s\in [0,1]$, $\gamma(s)$ solves
			\begin{equation*}
			A\gamma(s)+B(q(s)+z(s))=0 \quad \forall \ i\in \left\{0,1\right\}.
			\end{equation*}
			where $(q(s),z(s))\in \mbox{int}^V(\mathscr{U}_{\mbox{\tiny{ad}}}\cap V)\times \mathscr{U}_{\mbox{\tiny{ad}}}$ are defined by:
			\begin{equation*}
			q(s)=(1-s)q^0+sq^1\quad\mbox{and}\quad 	z(s)=(1-s)z^0+sz^1\quad \forall \ s\in [0,1].
			\end{equation*}
			The rest of the proof is divided into two steps.\\
			\textit{Step 1} \ \textbf{Show that there exists $\delta>0$, such that for each $s\in [0,1]$, $q(s)+B^V(0,\delta)\subset \mbox{int}^V(\mathscr{U}_{\mbox{\tiny{ad}}}\cap V)$, where $B^V(0,\delta)$ denotes the closed ball in $V$, centered at $0$ and of radius $\delta$.}\\
			
			Define
			\begin{equation}\label{th_6_eq_24}
			f(s)=\inf_{y\in V\setminus \mbox{int}^V(\mathscr{U}_{\mbox{\tiny{ad}}}\cap V)}\|q(s)-y\|_V,\quad s\in [0,1].
			\end{equation}
			One can check that $f$ is Lipschitz continuous over the compact interval $[0,1]$. Then, by Weierstrass' Theorem, we have that
			\begin{equation*}
			\min_{s\in [0,1]}f(s)>0.
			\end{equation*}
			Choose $0<\delta< \min_{s\in [0,1]}f(s)$. Hence, by \eqref{th_6_eq_24}, it follows that, for each $s\in [0,1]$,
			\begin{equation*}
			q(s)+B^V(0,\delta)\subset \mbox{int}^V(\mathscr{U}_{\mbox{\tiny{ad}}}\cap V),
			\end{equation*}
			as required.\\
			\textit{Step 2} \ \textbf{Conclusion.}\\
			Let $C>0$ be given by Lemma \ref{lemma_10}. Let $\delta >0$ be given by Step $1$. Choose $N_0\in\mathbb{N}\setminus \left\{0\right\}$ such that
			\begin{equation}\label{th_6_eq_3}
			N_0>\frac{2C\|y_0-y_1\|_H}{\delta}.
			\end{equation}
			For each $k\in \left\{0,\dots,N_0\right\}$, define:
			\begin{equation}\label{th_6_eq_16}
			y_k=\left(1-\frac{k}{N_0}\right)y_0+\frac{k}{N_0}y_1\quad\mbox{and}\quad 		u_k=\left(1-\frac{k}{N_0}\right)\overline{u}^0+\frac{k}{N_0}\overline{u}^1.
			\end{equation}
			It is clear that, by \eqref{th_6_eq_2}, for each $k\in \left\{0,\dots,N_0-1\right\}$,
			\begin{equation}\label{th_6_eq_16_bis}
			\|y_k-y_{k+1}\|_H=\frac{1}{N_0}\|y_0-y_1\|_H\quad\mbox{and}\quad u_k-q\left(\frac{k}{N_0}\right)\in \mathscr{U}_{\mbox{\tiny{ad}}}.
			\end{equation}
			Arbitrarily fix $k\in \left\{0,\dots, N_0-1\right\}$. Take $\eta_0=y_k-y_{k+1}$ and $\overline{v}^0=u_k-u_{k+1}$. Then, we apply Lemma \ref{lemma_10},  getting a control $w_k\in L^{\infty}(1,T_0+1;V)$ such that
			\begin{equation}\label{th_6_eq_17}
			y(T_0+1;y_k-y_{k+1},\hat{v}_k)=0
			\end{equation}
			and
			\begin{equation}\label{th_6_eq18}
			\|w_k\|_{L^{\infty}(1,T_0+1;V)}\leq C\|y_k-y_{k+1}\|_H,
			\end{equation}
			where
			\begin{equation}\label{th_6_eq20}\hat{v}_k(t)=
			\begin{cases}
			\rho(t)(u_k-u_{k+1}) \quad &t\in  (0,1]\\
			w_k(t)\quad &t\in (1,T_0+1).
			\end{cases}
			\end{equation}
			Define
			\begin{equation}\label{th_6_eq22}
			v_k(t)=		\begin{cases}
			\rho(t)(u_k-u_{k+1})+u_{k+1} \quad &t\in (0,1]\\
			w_k(t)+u_{k+1}\quad &t\in (1,T_0+1).
			\end{cases}
			\end{equation}
			At the same time, by \eqref{th_6_eq_1} and \eqref{th_6_eq_16}, we have
			\begin{equation*}
			Ay^{k+1}+Bu_{k+1}=0\quad \mbox{and}\quad y(T_0+1;y_{k+1},u_{k+1})=y_{k+1}.
			\end{equation*}
			The above, together with \eqref{th_6_eq_17}, \eqref{th_6_eq20} and \eqref{th_6_eq22}, yields
			\begin{eqnarray}\label{th_6_eq_29}
			y(T_0+1;y_k,v_k)&=&y(T_0+1;y_k-y_{k+1},\hat{v}_k)+y(T_0+1;y_{k+1},u_{k+1})\nonumber\\
			&=&y_{k+1}.
			\end{eqnarray}
			
			Next, we claim that
			\begin{equation}\label{th_6_eq_30}
			v_k(t)\in \mathscr{U}_{\mbox{\tiny{ad}}}\quad \mbox{for a.e.} \ t\in (0,T_0+1).
			\end{equation}
			To this end, by \eqref{def_W} and since $\mathscr{U}_{\mbox{\tiny{ad}}}$ is a convex cone, we have
			\begin{equation}\label{th_6_eq24}
			\mathscr{W}\hspace{0.3 cm}\mbox{is convex and}\hspace{0.3 cm}\mathscr{W}\subset \mathscr{U}_{\mbox{\tiny{ad}}}.
			\end{equation}
			By \eqref{lemma_10_hypo}, $0\leq \rho(t)\leq 1$ for all $t\in \mathbb{R}$. Then, by \eqref{th_6_eq22} an \eqref{th_6_eq24}, it follows that, for a.e $t\in (0,1)$,
			\begin{equation*}
			v_k(t)=\rho(t)u_k+(1-\rho(t))u_{k+1}\in \rho(t)\mathscr{W}+(1-\rho(t))\mathscr{W}\subset \mathscr{W}\subset \mathscr{U}_{\mbox{\tiny{ad}}}.
			\end{equation*}
			At this stage, to show \eqref{th_6_eq_30}, it remains to prove that
			\begin{equation}\label{th_6_eq_36}
			v_k(t)\in \mathscr{U}_{\mbox{\tiny{ad}}}\quad \mbox{for a.e.} \ t\in (1,T_0+1).
			\end{equation}
			Take $t\in (1,T_0+1)$. By \eqref{th_6_eq18}, \eqref{th_6_eq_16_bis} and \eqref{th_6_eq_3}, we have
			\begin{equation*}
			\|w_k(t)\|_V\leq \frac{C}{N_0}\|y_0-y_1\|_H\leq \delta/2.
			\end{equation*}
			From this and Step 1, it follows
			\begin{equation*}
			w_k(t)+q\left(\frac{k+1}{N_0}\right)\in \mbox{int}^V(\mathscr{U}_{\mbox{\tiny{ad}}}\cap V).
			\end{equation*}
			By this, \eqref{th_6_eq_16_bis}, \eqref{th_6_eq22} and \eqref{def_W}, we get, for a.e. $t$ in $(1,T_0+1)$,
			\begin{eqnarray}
			v_k(t)&=&w_k(t)+u_{k+1}\nonumber\\
			&=&w_k(t)+q\left(\frac{k+1}{N_0}\right)+\left(u^{k+1}-q\left(\frac{k+1}{N_0}\right)\right)\nonumber\\
			&\in & \mbox{int}^V(\mathscr{U}_{\mbox{\tiny{ad}}}\cap V)+\mathscr{U}_{\mbox{\tiny{ad}}}\nonumber\\
			&=&\mathscr{W}.\nonumber
			\end{eqnarray}
			From this and \eqref{th_6_eq24}, we are led to \eqref{th_6_eq_36}. Therefore, the claim \eqref{th_6_eq_30} is true.
			
			Finally, define
			\begin{equation*}
			u(t)=v_k(t-k(T_0+1)),\hspace{0.60 cm} \forall \ t\in [k(T_0+1),(k+1)(T_0+1)), \hspace{0.60 cm} k\in \left\{0,\dots,N_0-1\right\}.
			\end{equation*}
			Then, from \eqref{th_6_eq_29} and \eqref{th_6_eq_30}, the conclusion of this theorem follows.
			\qed
		\end{proof}
		
		In subsections \ref{subsec:3.1} and \ref{subsec:4.1}, we apply the above Theorem to prove Theorem \ref{th_2} and Theorem \ref{th_1} respectively. In particular, 
		\begin{itemize}
			\item for internal control,
			\begin{equation*}
			\mathscr{U}_{\mbox{\tiny{ad}}}=\left\{u\in L^2(\omega) \ | \ u\geq 0, \ \mbox{a.e.} \ {\omega}\right\};
			\end{equation*}
			\item for boundary control,
			\begin{equation*}
			\mathscr{U}_{\mbox{\tiny{ad}}}=\left\{u\in L^2(\Gamma) \ | \ u\geq 0, \ \mbox{a.e.} \ {\Gamma}\right\}.
			\end{equation*}
		\end{itemize}
		Then, in both cases, $\mathscr{U}_{\mbox{\tiny{ad}}}$ is closed convex cone with vertex at $0$.
		
		Nevertheless, the above techniques can be adapted in a wide variety of contexts.

		\subsection{Controllability between trajectories}
		\label{subsec:2.2}

		In this subsection, we study the constrained controllability for some general states lying on trajectories of the system with possibly nonzero time derivative. Recall $s$ and $V$ are given by \eqref{s and V}. Before introducing our main result, we assume:
		
		\hspace{0.03 cm}$(H_1^{\prime})$ the system \eqref{state_equation_abstract} is smoothly controllable in time $T_0$ for some $T_0>0$.
		
		\hspace{0.03 cm}$(H_2^{\prime})$ the set $\mathscr{U}_{\mbox{\tiny{ad}}}$ is a closed and convex and $\mbox{int}^V(\mathscr{U}_{\mbox{\tiny{ad}}}\cap V)\neq \varnothing$, where $\mbox{int}^V$ denotes
		
		\hspace{0.80 cm}the interior set in the topology of $V$;
		
		\hspace{0.03 cm}$(H_3^{\prime})$ the operator $A$ generates a $C_0$-group $\left\{\mathbb{T}_t\right\}_{t\in\mathbb{R}}$ over $H$ and $\|\mathbb{T}_t\|_{\mathscr{L}(H,H)}=1$ for all $t\in\mathbb{R}$. Furthermore, $A$ is invertible from $D(A)$ to $H$, with continuous inverse.
		
		The main result of this subsection is the following. The notation $y(\cdot;y_0,u)$ stands for the solution of the abstract controlled equation \eqref{state_equation_abstract} with control $u$ and initial data $y_0$.

		\begin{theorem}\label{th_8}
			Assume $(H_1^{\prime})$, $(H_2^{\prime})$ and $(H_3^{\prime})$ hold. Let $\overline{y}_i\in C^s(\mathbb{R};H)$ be solutions to \eqref{state_equation_abstract} with controls $\overline{u}^i\in L^2_{loc}(\mathbb{R};U)$ for $i=0,1$. Assume $\overline{u}^i(t)\in \mathscr{U}_{\mbox{\tiny{ad}}}$ for a.e. $t\in \mathbb{R}$. Let $\tau_0$, $\tau_1\in \mathbb{R}$. Then, there exists $T>0$ and $u\in L^2(0,T;U)$ such that
			\begin{itemize}
				\item $y(T;\overline{y}_0(\tau_0),u)=\overline{y}_1(\tau_1)$;
				\item $u(t)\in \mathscr{U}_{\mbox{\tiny{ad}}}$ for a.e. $t\in (0,T)$.
			\end{itemize}
		\end{theorem}
		
		\begin{remark}
			(i) Roughly, Theorem \ref{th_8} addresses the constrained controllability for all initial data $y_0$ and final target $y_1$, with $y_0,\ y_1\in E$, where
			\begin{equation*}
			E=\bigg\{y(\tau)\ \Big| \ \tau\in\mathbb{R}, \ y\in C^s(\mathbb{R};H)\quad\mbox{and}\quad \exists \ u\in L^2_{loc}(\mathbb{R};U) ,\bigg.
			\end{equation*}
			\begin{equation*}
			\bigg.\mbox{with}\quad u(t)\in\mathscr{U}_{\mbox{\tiny{ad}}}\quad\mbox{a.e.} \ t\in \mathbb{R}\quad \mbox{s.t.}\quad\frac{d}{dt}y(t)=Ay(t)+Bu(t),\quad t\in\mathbb{R} \bigg\}.
			\end{equation*}
			By Lemma \ref{lemma_semigroup_4}, one can check that
			\begin{equation*}
			\left\{y(\tau;0,u) \ \Big| \ \tau\in\mathbb{R},\, u\in C^s(\mathbb{R},\mathscr{U}_{\mbox{\tiny{ad}}}), \frac{d^j}{dt^j}u(0)=0, \ j=0,\dots,s \right\}\subset E.
			\end{equation*}
			Furthermore, we observe that such set $E$ includes some non-steady states.\\
			(ii) There are at least two differences between Theorem \ref{th_6} and Theorem \ref{th_8}. First of all, Theorem \ref{th_6} studies constrained controllability for some steady states, whereas Theorem \ref{th_8} can deal with constrained controllability for some non-steady states (see (i) of this remark). Secondly, in Theorem \ref{th_8} the controls $\overline{u}^i$ ($i=0,1$) defining the initial datum $\overline{y}^0(\tau_0)$ and final target $\overline{y}^1(\tau_1)$ are required to fulfill the constraint
			\begin{equation*}
			\overline{u}^i(t)\in \mathscr{U}_{\mbox{\tiny{ad}}}, \quad\mbox{a.e.} \ t\in\mathbb{R}, \ i=0,1,
			\end{equation*}
			while $\overline{u}^i$ in Theorem \ref{th_6} is required to be in $\mathscr{W}\subsetneq \mathscr{U}_{\mbox{\tiny{ad}}}$. (Then, in Theorem \ref{th_8} we have weakened the constraints on $\overline{u}^i$. In particular, we are able to apply Theorem \ref{th_8} to the wave system with nonnegative controls with final target $\overline{y}^1\equiv 0$.)
		\end{remark}

		Before proving Theorem \ref{th_8}, we show a preliminary lemma. Note that such Lemma works with any \textit{contractive} semigroup. In particular, it holds both for wave-like and heat-like systems. A similar result was proved in \cite{phung2007existence} and \cite{loheac2013maximum}. For the sake of completeness, we provide the proof of the aforementioned lemma in the Appendix.
		
		\begin{lemma}[Null Controllability by small controls]\label{lemma_stair_case_6}
			Assume that $A$ generates a
			contractive $C_0$-semigroup $(\mathbb{T}_t)_{t\in\mathbb{R}^+}$ over $H$.  Suppose that ($H_1^{\prime}$) holds. Let $\varepsilon >0$ and $\eta_0\in D(A^s)$. Then, there exists $\overline{T}=\overline{T}(\varepsilon,\|\eta_0\|_{D(A^{s})})>0$ such that, for any $T\geq \overline{T}$, there exists a control $v\in L^{\infty}((0,T);V)$ such that
			\begin{itemize}
				\item $y(T;\eta_0,v)=0$;
				\item $\|v\|_{L^{\infty}(\mathbb{R}^+;V)}\leq \varepsilon$.
			\end{itemize}
		\end{lemma}
		The proof of the Lemma above is given in the Appendix.

		We are now ready to prove Theorem \ref{th_8}.
		
		With respect to Theorem \ref{th_5} we have weakened the constraints on the controls defining the initial and final trajectories. Then, a priori, we have lost the room for oscillations needed in the proof of that Theorem. We shall see how to recover this by modifying the initial and final trajectories away from the initial and final data (see figure \ref{INDAM 2018 proof 1}, figure \ref{INDAM 2018 proof 2} and figure \ref{INDAM 2018 proof 3}).
		
		\begin{figure}[htp]
			\begin{center}
				\includegraphics[width=8cm]{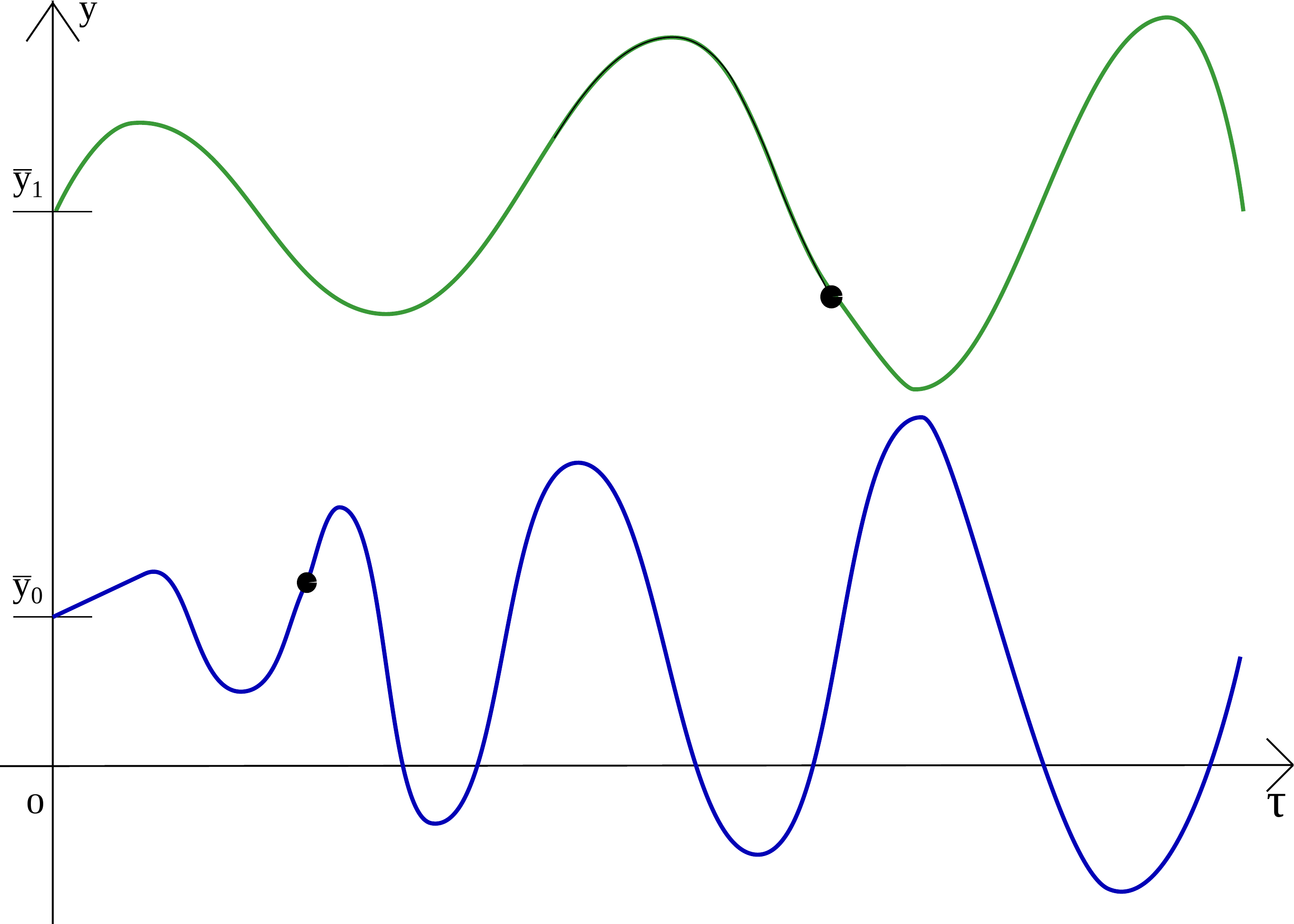}\\
				\caption{The two original trajectories. The time $\tau$ parameterizing the trajectories is just a parameter
					independent of the control time $t$.}\label{INDAM 2018 proof 1}
			\end{center}
		\end{figure}
		\begin{figure}[htp]
			\begin{center}
				\includegraphics[width=8cm]{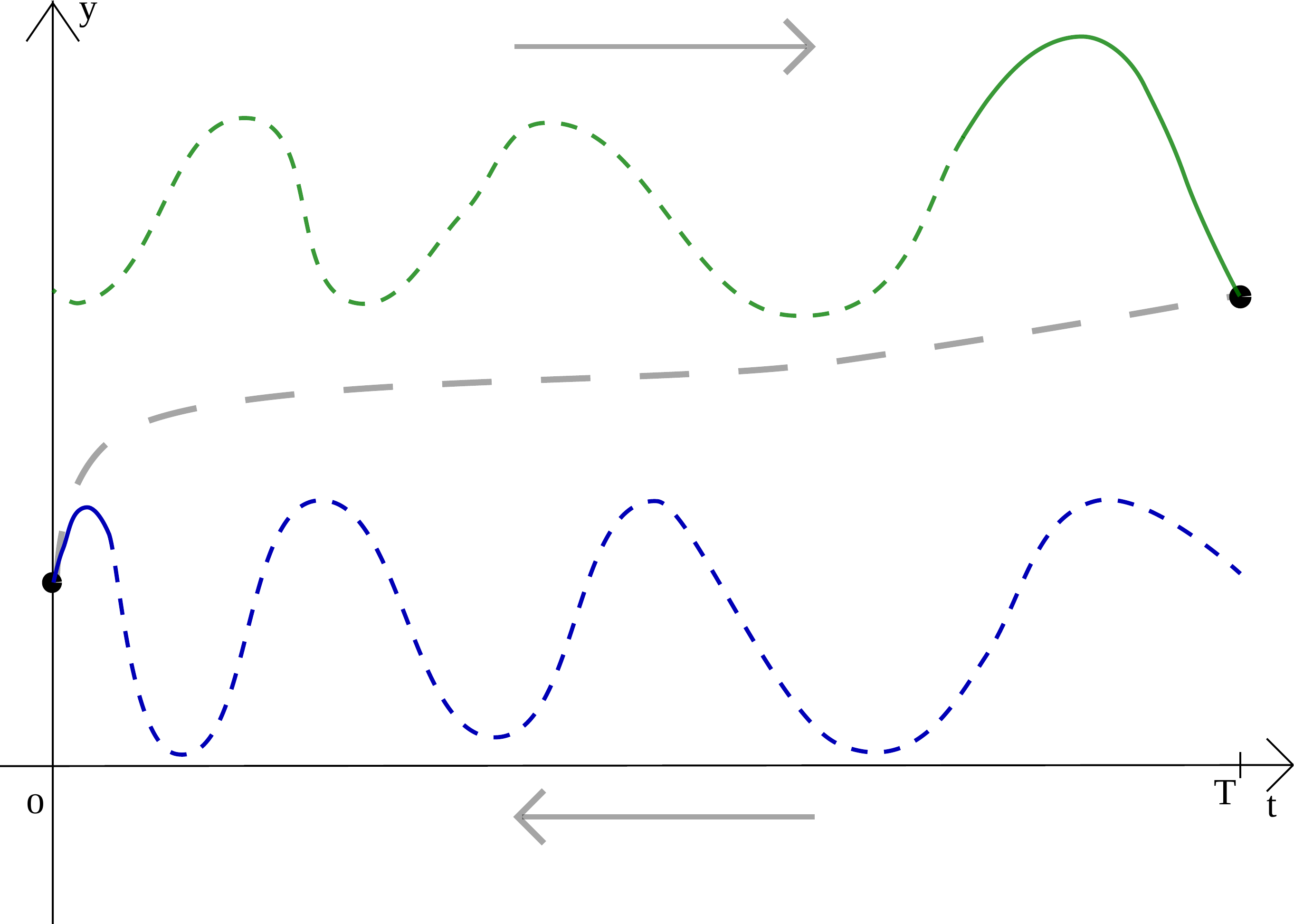}\\
				\caption{The new trajectories to be linked, now synchronized with the control time $t$. Note that 1) we have translated the time parameter defining the trajectories and 2) we have modified them away from the initial and the final data, to apply Lemma \ref{lemma_stair_case_6}. The new initial trajectory is represented in blue, while the new final trajectory is drawn in green. The modified part is dashed. Following the notation of the  proof of Theorem \ref{th_8}, the new initial trajectory is $y(\cdot;\hat{u}^0,\overline{y}_0(\tau_0))$, while the new final trajectory is $\varphi_T$.}\label{INDAM 2018 proof 2}
			\end{center}
		\end{figure}
		\begin{figure}[htp]
			\begin{center}
				\includegraphics[width=8cm]{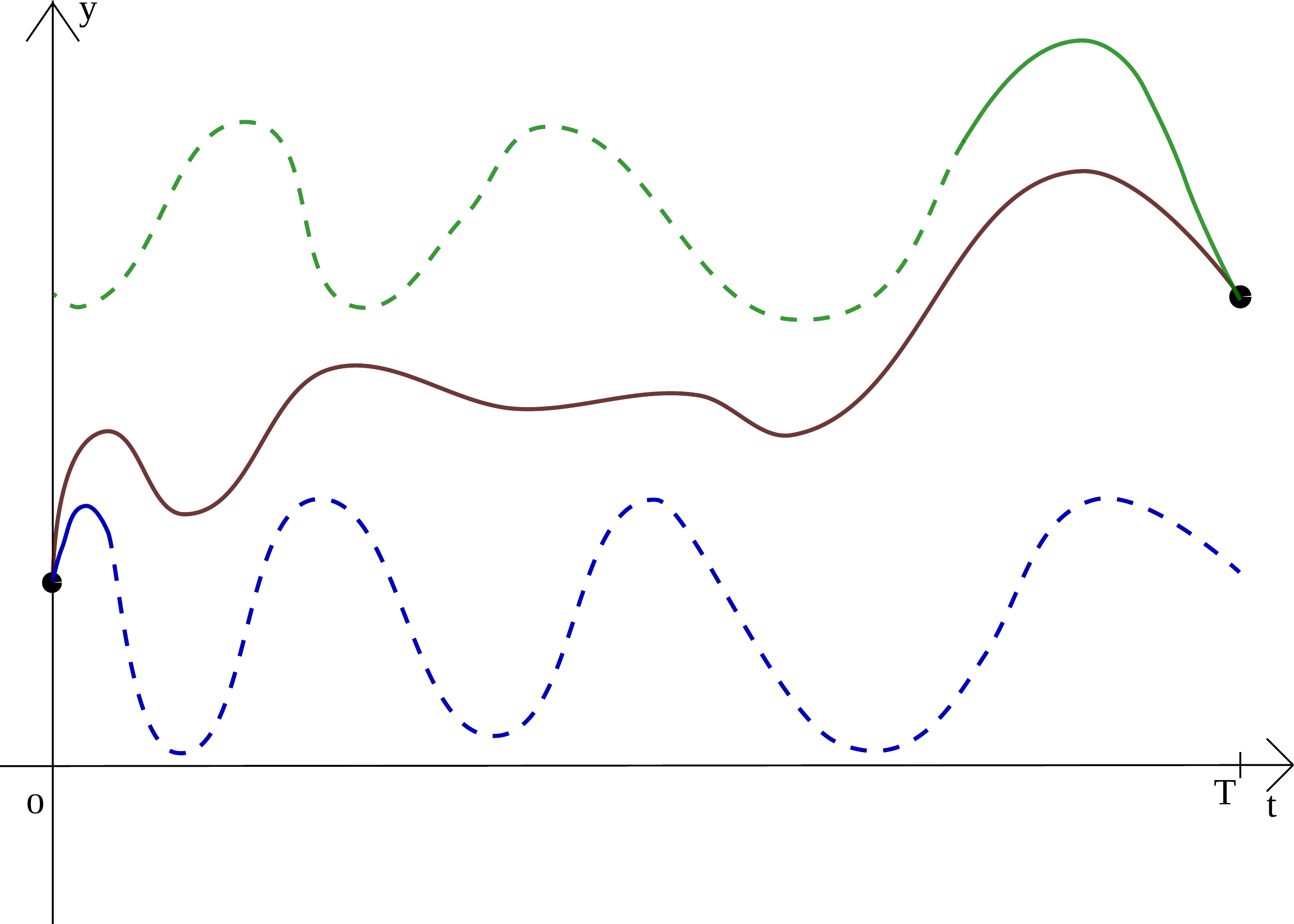}\\
				\caption{The new trajectories linked by the controlled trajectory $y$, pictured in red. As in figure \ref{INDAM 2018 proof 2}, the new initial trajectory is drawn in blue, while the new final trajectory is represented in green.}\label{INDAM 2018 proof 3}
			\end{center}
		\end{figure}

		\begin{proof}[Proof of Theorem \ref{th_8}]
			The main strategy of proof is the following:
			\begin{itemize}
				\item[(i)] we reduce the constrained controllability problem (with initial data $\overline{y}_0(\tau_0)$ and final target $\overline{y}_1(\tau_1)$) to another controllability problem (with initial datum $\hat{y}_0$ and final target $0$);
				\item[(ii)]  we solve the latter controllability problem by constructing two controls. The first control is used to improve the regularity of the solution. The second control is small in a regular space and steers the system to rest.
			\end{itemize}

			\textit{Step 1} \ \textbf{The part (i) of the above strategy.}\\
			For each $T>0$, we aim to define a new trajectory with the final state $\overline{y}_1(\tau_1)$ as value at time $t=T$. Choose a smooth function $\zeta\in C^{\infty}(\mathbb{R})$ such that
			\begin{equation}\label{def_zeta}
			\zeta\equiv 1 \ \mbox{over} \ \left(-\frac12,\frac12\right)\quad \mbox{and} \quad \mbox{supp}(\zeta)\subset\subset (-1,1).
			\end{equation}	
			Take $\sigma\in \mbox{int}^V(\mathscr{U}_{\mbox{\tiny{ad}}}\cap V)$. Arbitrarily fix $T>1$. Define a control
			\begin{equation}\label{th8_eq_6}
			\hat{u}_{T}^1(t)=\zeta(t-T)\overline{u}^1(t-T+\tau_1)+(1-\zeta(t-T))\sigma.
			\end{equation}
			We denote by $\varphi_T$ the unique solution to the problem
			\begin{equation}\label{th8_eq_-6}
			\begin{cases}
			\frac{d}{dt}\varphi(t)=A\varphi(t)+B\hat{u}_T^1(t)& t\in \mathbb{R}\\
			\varphi(T)=\overline{y}_1(\tau_1).\\
			\end{cases}
			\end{equation}
			In what follows, we will construct two controls which sends $\overline{y}_0(\tau_0)-\varphi_T(0)$ to $0$ in time $T$, which is part (ii) of our strategy. Recall that $\rho$ is given by \eqref{lemma_10_hypo}. We define 
			\begin{equation*}
			\hat{u}^0(t)=\rho(t)\overline{u}^0(t+\tau_0)+(1-\rho(t))\sigma \quad t\in\mathbb{R}.
			\end{equation*}
			\textit{Step 2} \ \textbf{Estimate of $\|y(1;\overline{y}_0(\tau_0)-\varphi_T(0),\hat{u}^0-\hat{u}^1_T)\|_{D(A^s)}$}\\
			We take the control $(\hat{u}^0-\hat{u}^1_T)\hspace{-0.1 cm}\restriction_{(0,1)}$ to be the first control mentioned in part (ii) of our strategy. In this step, we aim to prove the following regularity estimate associated with this control:
			there exists a constant $C>0$ independent of $T$ and $\sigma$ such that
			\begin{equation}\label{th_8_eq_20}
			\|y(1;\overline{y}_0(\tau_0)-\varphi_T(0),\hat{u}^0-\hat{u}^1_T)\|_{D(A^s)}
			\end{equation}
			\begin{equation*}
			\leq C\left[\|\overline{y}_0\|_{C^s([\tau_0,\tau_0+1];H)}+\|\overline{y}_1\|_{C^s([\tau_1-1,\tau_1];H)}+\|\sigma\|_{U}\right].
			\end{equation*}
			
			To begin, we introduce $\psi$ the solution to
			\begin{equation}\label{th8_eq_59}
			A\psi+B\sigma=0.
			\end{equation}
			First, we have that
			\begin{eqnarray}\label{th8_eq_60}
			&\;&y(1;\overline{y}(\tau_0)-\varphi_T(0),\hat{u}^0-\hat{u}_T^1)\nonumber\\
			&=&y(1;\overline{y}(\tau_0),\hat{u}^0)-y(1;\varphi_T(0),\hat{u}^1)\nonumber\\
			&=&[y(1;\overline{y}(\tau_0),\hat{u}^0)-\psi]-[y(1;\varphi_T(0),\hat{u}_T^1)-\psi]\nonumber\\
			&=&y(1;\overline{y}(\tau_0)-\psi,\hat{u}^0-\sigma)-y(1;\varphi_T(0)-\psi,\hat{u}_T^1-\sigma).
			\end{eqnarray}
			To estimate \eqref{th_8_eq_20}, we need to compute the norms of the last two terms in \eqref{th8_eq_60}, in the space $D(A^s)$. We claim that there exists $C_1>0$ (independent of $T$ and $\sigma$) such that
			\begin{equation}\label{th8_eq_66}
			\|y(1;\overline{y}(\tau_0)-\psi,\hat{u}^0-\sigma)\|_{D(A^s)}\leq C_1\left(\|\overline{y}_0\|_{C^s([\tau_0,\tau_0+1];H)}+\|\sigma\|_{U}\right).
			\end{equation}
			To this end, we show that
			\begin{equation}\label{th_8_eq90}
			y(t;\overline{y}(\tau_0)-\psi,\hat{u}^0-\sigma)=\rho(t)(\overline{y}^0(t+\tau_0)-\psi)+\eta_2(t),\quad t\in \mathbb{R},
			\end{equation}
			where $\eta_2$ solves
			\begin{equation}\label{th8_eq_100}
			\begin{cases}
			\frac{d}{dt}\eta_2(t)=A\eta_2(t)-\rho^{\prime}(\overline{y}(t+\tau_0)-\psi)& t\in \mathbb{R}\\
			\eta_2(0)=0.\\
			\end{cases}
			\end{equation}
			Indeed,
			\begin{eqnarray}\label{th8_eq_101}
			&\;&\frac{d}{dt}\left[\rho(t)(\overline{y}^0(t+\tau_0)-\psi)+\eta_2(t)\right]\nonumber\\
			&=&\rho(t)(A\overline{y}^0(t+\tau_0)+B\overline{u}^0(t+\tau_0))+\rho^{\prime}(t)(\overline{y}^0(t+\tau_0)-\psi)\nonumber\\
			&&+A\eta_2(t)-\rho^{\prime}(t)(\overline{y}^0(t+\tau_0)-\psi)\nonumber\\
			&=&A(\rho(t)\overline{y}^0(t+\tau_0)+\eta_2(t))+B\left(\rho(t)\overline{u}^0(t+\tau_0)\right)\nonumber\\
			&=&A(\rho(t)(\overline{y}^0(t+\tau_0)-\psi)+\eta_2(t))+\rho(t)A\psi+B\left(\rho(t)\overline{u}^0(t+\tau_0)\right)\nonumber\\
			&=&A(\rho(t)(\overline{y}^0(t+\tau_0)-\psi)+\eta_2(t))+B\left(\rho(t)(\overline{u}^0(t+\tau_0)-\sigma)\right)\nonumber\\
			&=&A(\rho(t)(\overline{y}^0(t+\tau_0)-\psi)+\eta_2(t))+B(\hat{u}^0(t)-\sigma).
			\end{eqnarray}
			At the same time, since $\rho(0)=1$, from \eqref{th8_eq_100}, it follows that
			\begin{equation*}
			\rho(t)(\overline{y}^0(t+\tau_0)-\psi)+\eta_2(t)\hspace{-0.1 cm}\restriction_{t=0}=\overline{y}^0(\tau_0)-\psi.
			\end{equation*}
			From this and \eqref{th8_eq_101}, we are led to \eqref{th_8_eq90}.
			
			Next, we will use \eqref{th_8_eq90} and \eqref{th8_eq_100} to prove \eqref{th8_eq_66}. To this end, since we assumed $\overline{y}^0\in C^s(\mathbb{R};H)$ and $\psi$ is independent of $t$, we get that
			\begin{equation*}
			\overline{y}^0(\cdot+\tau_0)-\psi\in C^s(\mathbb{R};H).
			\end{equation*}
			By this, we apply Lemma \ref{lemma_semigroup_4} obtaining the existence of $\hat{C}_1>0$ (independent of $T$ and $\sigma$) such that
			\begin{equation}\label{th8_eq_120}
			\|\eta_2(1)\|_{D(A^s)}\leq \hat{C}_1\left(\|\overline{y}^0\|_{C^s([\tau_0,\tau_0+1];H)}+\|\psi\|_H\right).
			\end{equation}
			At the same time, since $\rho(1)=0$ (see \eqref{lemma_10_hypo}), by \eqref{th_8_eq90}, we have that
			\begin{equation*}
			y(1;\overline{y}(T_0)-\psi,\hat{u}^0-\sigma)=\eta_2(1).
			\end{equation*}	 
			This, together with \eqref{th8_eq_120} and \eqref{th8_eq_59}, yields \eqref{th8_eq_66}.
			
			At this point, we estimate the norm of the second term in \eqref{th8_eq_60} in the space $D(A^s)$, namely we prove the existence of $C_2>0$ (independent of $T$ and $\sigma$) such that
			\begin{equation}\label{th8_eq121}
			\|y(1;\varphi_T(0)-\psi,\hat{u}_T^1-\sigma)\|_{D(A^s)}\leq C_2\left[\|\overline{y}^1\|_{C^s([\tau_1-1,\tau_1];H)}+\|\sigma\|_{U}\right].
			\end{equation}
			To this end, as in the proof of \eqref{th_8_eq_20}, we get that
			\begin{equation}\label{th8_eq122}
			y(t;\varphi_T(0)-\psi,\hat{u}_T^1-\sigma)=\zeta(t-T)(\overline{y}^1(t-T+\tau_1)-\psi)+\tilde{\eta}_2(t), \quad t\in\mathbb{R},
			\end{equation}
			where $\tilde{\eta}_2$ solves
			\begin{equation}\label{th8_eq123}
			\begin{cases}
			\frac{d}{dt}\tilde{\eta}_2(t)=A\tilde{\eta}_2(t)-\zeta^{\prime}(t-T)(\overline{y}^1(t-T+\tau_1)-\psi)& t\in \mathbb{R}\\
			\tilde{\eta}_2(T)=0.\\
			\end{cases}
			\end{equation}
			We will use \eqref{th8_eq122} and \eqref{th8_eq123} to prove \eqref{th8_eq121}. Indeed, set
			\begin{equation*}
			\hat{\eta}(t)=\tilde{\eta}_2(T-t).
			\end{equation*}
			By definition of $\hat{\eta}$, we have
			\begin{equation}\label{th8_eq124}
			\begin{cases}
			\frac{d}{dt}\hat{\eta}(t)=-A\hat{\eta}(t)+\zeta^{\prime}(-t)(\overline{y}^1(\tau_1-t)-\psi)& t\in \mathbb{R}\\
			\hat{\eta}(0)=0.\\
			\end{cases}
			\end{equation}
			Since we have assumed
			$\overline{y}^1\in C^s(\mathbb{R}, H)$ and $\psi$ is independent of $t$ (see \eqref{th8_eq_59}), we have
			\begin{equation*}
			\overline{y}^1-\psi\in C^s(\mathbb{R};H).
			\end{equation*}
			Recall that $\zeta(t)\equiv 1$ in $\left(-\frac12,\frac12\right)$ (see \eqref{def_zeta}). Then, $\zeta^{\prime}(t)= 0$, for each $t\in \left(-\frac12,\frac12\right)$. Now, by hypothesis $(H_3^{\prime})$, $A$ generates a group of operators.
			Hence, we can apply Lemma \ref{lemma_semigroup_4} to \eqref{th8_eq124} getting the existence of $\tilde{C}_2>0$ (independent of $T$ and $\sigma$) such that
			\begin{equation*}
			\|\hat{\eta}(1)\|_{D(A^s)}\leq \tilde{C}_2\left(\|\overline{y}^1\|_{C^s([\tau_1-1,\tau_1];H)}+\|\psi\|_H\right),
			\end{equation*}
			whence
			\begin{equation}\label{th8_eq130}
			\|\tilde{\eta}_2(T-1)\|_{D(A^s)}\leq \tilde{C}_2\left(\|\overline{y}^1\|_{C^s([\tau_1-1,\tau_1];H)}+\|\psi\|_H\right).
			\end{equation}
			At the same time, by ($H_3^{\prime}$) and some computations, we have that
			\begin{equation*}
			\|\mathbb{T}_t\|_{\mathscr{L}(D(A^s),D(A^s))}=1, \quad \mbox{for each} \ t\in\mathbb{R}.
			\end{equation*}
			Since $\zeta(t-T)=0$, for each $t\in [0,T-1]$ (see \eqref{def_zeta}), the above, together with \eqref{th8_eq122} and \eqref{th8_eq123}, yields
			\begin{equation*}
			\|y(1;\varphi_T(0)-\psi,\hat{u}_T^1-\sigma)\|_{D(A^s)}=\|\tilde{\eta}_2(1)\|_{D(A^s)}
			=\|\tilde{\eta}_2(T-1)\|_{D(A^s)}.
			\end{equation*}
			This, together with \eqref{th8_eq130} and \eqref{th8_eq_59}, leads to \eqref{th8_eq121}.\\
			\textit{Step 3} \ \textbf{Conclusion.}\\
			In this step, we will first construct the second control mentioned in part (ii) of our strategy. Then we put together the first and second controls (mentioned in part (ii)) to get the conclusion.
			
			By \eqref{th8_eq121},
			\begin{equation*}
			\|y(1;\varphi_T(0)-\psi,\hat{u}_T^1-\sigma)\|_{D(A^s)}\leq C_2\left[\|\overline{y}^1\|_{C^s([\tau_1-1,\tau_1];H)}+\|\sigma\|_{U}\right].
			\end{equation*}
			The above estimate is independent of $T$. Then for each $T>0$, by  Lemma \ref{lemma_stair_case_6}, there exists
			$$\overline{T}=\overline{T}(\sigma,\|\overline{y}^0\|_{C^s([\tau_0,\tau_0+1];H)},\|\overline{y}^1\|_{C^s([\tau_1-1,\tau_1];H)})>0$$ and $w_T\in L^{\infty}(\mathbb{R}^+;V)$ such that
			\begin{equation}\label{th8_eq158}
			\begin{cases}
			\frac{d}{dt}z(t)=Az(t)+Bw_T(t)& t\in (1,\overline{T})\\
			z(1)=y(1;\overline{y}(\tau_0)-\varphi_T(0),\hat{u}^0-\hat{u}_T^1),\quad z(\overline{T})=0\\
			\end{cases}
			\end{equation}
			and
			\begin{equation}\label{th8_eq179}
			\|w_T\|_{L^{\infty}(1,\overline{T};V)}\leq \frac12 \inf_{y\in V\setminus \mbox{int}^V(\mathscr{U}_{\mbox{\tiny{ad}}}\cap V)}\|\sigma-y\|_V.
			\end{equation}
			Note that the last constant is positive, because $\sigma$ is taken from $\mbox{int}^V(\mathscr{U}_{\mbox{\tiny{ad}}})$.
			Choose\\
			$T=\overline{T}+1$. Define a control:
			\begin{equation}\label{th8_eq180}
			v=\begin{cases}
			\hat{u}^0(t) \quad &t\in \ (0,1)\\
			w_T(t)+\hat{u}_T^1(t) \quad &t\in \ (1,\overline{T})\\
			\hat{u}_T^1(t) \quad &t\in \ (\overline{T},\overline{T}+1).\\
			\end{cases}
			\end{equation}
			We aim to show that
			\begin{equation}\label{th8_eq_200}
			y(\overline{T}+1;\overline{y}^0(\tau_0),v)=\overline{y}^0(\tau_1)\quad\mbox{and}\quad v(t)\in \mathscr{U}_{\mbox{\tiny{ad}}} \ \mbox{a.e.} \ t\in \ (1,\overline{T}+1).
			\end{equation}
			To this end, by \eqref{th8_eq180}, \eqref{th8_eq158} and \eqref{th8_eq_-6}, we get that
			\begin{eqnarray}
			y(\overline{T}+1;\overline{y}^0(\tau_0),v)&=&y(\overline{T}+1;\overline{y}^0(\tau_0)-\varphi_T(0),v-\hat{u}^1_T)+ y(\overline{T}+1;\varphi_T(0),\hat{u}^1_T)\nonumber\\
			&=&\mathbb{T}_1(z_T(\overline{T}))+\varphi_T(\overline{T}+1)\nonumber\\
			&=&\overline{y}^1(\tau_1)\nonumber.
			\end{eqnarray}
			This leads to the first conclusion of \eqref{th8_eq_200}. It remains to show the second condition in \eqref{th8_eq_200}. Arbitrarily fix $t\in (0,1)$. By \eqref{th8_eq180} and \eqref{th8_eq121}, we have
			\begin{equation*}
			v(t)=\rho(t)\overline{u}^0(t+\tau_0)+(1-\rho(t))\sigma
			\end{equation*}
			\begin{equation*}
			\in \rho(t)\mathscr{U}_{\mbox{\tiny{ad}}}+(1-\rho(t))\mathscr{U}_{\mbox{\tiny{ad}}}\subset \mathscr{U}_{\mbox{\tiny{ad}}}.
			\end{equation*}
			Choose also an arbitrary $s\in (1,\overline{T})$. By \eqref{th8_eq180}, \eqref{th8_eq179} and \eqref{th8_eq_6}, we obtain
			\begin{equation*}
			v(s)=w(s)+(1-\zeta(s-\overline{T}-1))\sigma+\zeta(s-\overline{T}-1)\overline{u}^1(s-\overline{T}-1+\tau_1)
			\end{equation*}
			\begin{equation*}
			=w(s)+\sigma \in \mbox{int}^V(\mathscr{U}_{\mbox{\tiny{ad}}}\cap V)\subset \mathscr{U}_{\mbox{\tiny{ad}}}.
			\end{equation*}
			Take any $t\in (\overline{T},\overline{T}+1)$. We find from \eqref{th8_eq180} and \eqref{th8_eq_6} that
			\begin{equation*}
			v(t)=\zeta(t-\overline{T}-1)\overline{u}^1(t-\overline{T}-1+\tau_1)+(1-\zeta(t-\overline{T}-1))\sigma
			\end{equation*}
			\begin{equation*}
			\in \zeta(t-\overline{T}-1)\mathscr{U}_{\mbox{\tiny{ad}}}+(1-\zeta(t-\overline{T}-1))\mathscr{U}_{\mbox{\tiny{ad}}} 
			\end{equation*}
			\begin{equation*}
			\subset \mathscr{U}_{\mbox{\tiny{ad}}}.
			\end{equation*}
			Therefore, we are led to the second conclusion of \eqref{th8_eq_200}. This ends the proof.
			\qed
		\end{proof}

		\section{Internal Control: Proof of Theorem \ref{th_2} and Theorem \ref{th_3}}
		\label{sec:3}
		
		The present section is organized as follows:
		\begin{itemize}
			\item Subsection \ref{subsec:3.1}: proof of Lemma \ref{Lemma_smooth_contr_wave_int} and Theorem \ref{th_2};
			\item Subsection \ref{subsec:3.2}: proof of Theorem \ref{th_3};
			\item Subsection \ref{subsec:3.3}: discussion of the issues related to the internal control touching the boundary.
		\end{itemize}
		
		\subsection{Proof of Theorem \ref{th_2}}
		\label{subsec:3.1}
		
		We now prove Theorem \ref{th_2} by employing Theorem \ref{th_6}.
		
		Firstly, we place our control system in the abstract framework introduced in section \ref{sec:2} and we prove that our control system is smoothly controllable (see Definition \ref{contr_2_abstract}).
		
		The free dynamics is generated by $A:D(A)\subset H\longrightarrow H$, where
		\begin{equation}\label{generator_internal_1}
		A=\begin{pmatrix}
		0&I\\
		-A_0&0\\
		\end{pmatrix},
		\hspace{0.5 cm}
		\begin{cases}
		H=H^1_0(\Omega)\times L^2(\Omega)\\
		D(A)=\left(H^2(\Omega)\cap H^1_0(\Omega)\right)\times H^1_0(\Omega).
		\end{cases}
		\end{equation}
		where $A_0=-\Delta+cI:H^2(\Omega)\cap H^1_0(\Omega)\subset L^2(\Omega)\longrightarrow L^2(\Omega)$. The control operator
		\begin{equation*}
		B(v)=\begin{pmatrix}
		0\\
		\chi v.\\
		\end{pmatrix}
		\end{equation*}
		defined from $U=L^2(\omega)$ to $H=H^1_0(\Omega)\times L^2(\Omega)$ is bounded, then admissible.

		\begin{lemma}\label{Lemma_smooth_contr_wave_int}
			In the above framework take $V=H^{s(n)}(\omega)$ and $s=s(n)=\lfloor{n/2}\rfloor+1$. Assume further $(\Omega,\omega_0,T^*)$ fulfills the Geometric Control Condition.  Then, the control system \eqref{wave_internal_1} is smoothly controllable in any time $T_0>T^*$.
		\end{lemma}
		The proof of this Lemma can be found in the reference \cite[Theorem 5.1]{ervedoza2010systematic}.
		
		We are now ready to prove Theorem \ref{th_2}.
		\begin{proof}[of Theorem \ref{th_2}]
			We choose as set of admissible controls:
			\begin{equation*}
			\mathscr{U}_{\mbox{\tiny{ad}}}=\left\{u\in L^2(\omega) \ | \ u\geq 0, \ \mbox{a.e.} \ \omega \right\}.
			\end{equation*}
			Then,
			\begin{equation}\label{strict_constrants_0}
			\bigcup_{\sigma >0}\left\{u\in L^2(\omega) \ | \ u\geq \sigma, \ \mbox{a.e.} \ {\omega}\right\}\subset 	\mathscr{W}.
			\end{equation}
			
			We highlight that, to prove \eqref{strict_constrants_0}, we need  $H^{s(n)}(\omega)\hookrightarrow C^0(\overline{\omega})$. For this reason, we have chosen $s(n)=\lfloor{n/2}\rfloor+1$.

			By Lemma \eqref{Lemma_smooth_contr_wave_int}, we have that the system is \textit{Smoothly Controllable} with $s=s(n)=\lfloor{n/2}\rfloor+1$ and $V=H^{s(n)}(\omega)$. Then, by Theorem \ref{th_6} we conclude.
			\qed\end{proof}
		
		\subsection{Proof of Theorem \ref{th_3}}
		\label{subsec:3.2}
		
		We prove now Theorem \ref{th_3}
		
		\begin{proof}[Proof of Theorem \ref{th_3}]
			As we have seen, our system fits the abstract framework. Moreover, we have checked in Lemma \ref{Lemma_smooth_contr_wave_int} that the system is \textit{Smoothly Controllable} with $s(n)=\lfloor{n/2}\rfloor+1$ and $V=H^{s(n)}(\omega)$. Furthermore, $\mbox{int}^V(\mathscr{U}_{\mbox{\tiny{ad}}}\cap V)\neq \varnothing$. Indeed, any constant $\sigma >0$ belongs to $\mbox{int}^V(\mathscr{U}_{\mbox{\tiny{ad}}}\cap V)$, since $H^{s(n)}(\omega)\hookrightarrow C^0(\overline{\omega})$. This is guaranteed by our choice of $s(n)=\lfloor{n/2}\rfloor+1$.
			
			Therefore, we are in position to apply Theorem \ref{th_8} and finish the proof.
			\qed
		\end{proof}

		\subsection{Internal controllability from a neighborhood of the boundary}
		\label{subsec:3.3}
		
		So far, we have assumed that the control is localized by means of a smooth cut-off function $\chi$ so that all its derivatives vanish on the boundary of $\Omega$. This implies that $\chi$ must be constant on any connected component of the boundary. This prevents us to localize the internal control in a region touching the boundary only on a subregion, as in figure \ref{interior_boundary}.
		
		\begin{figure}[htp]
			\begin{center}
				\includegraphics[width=6cm]{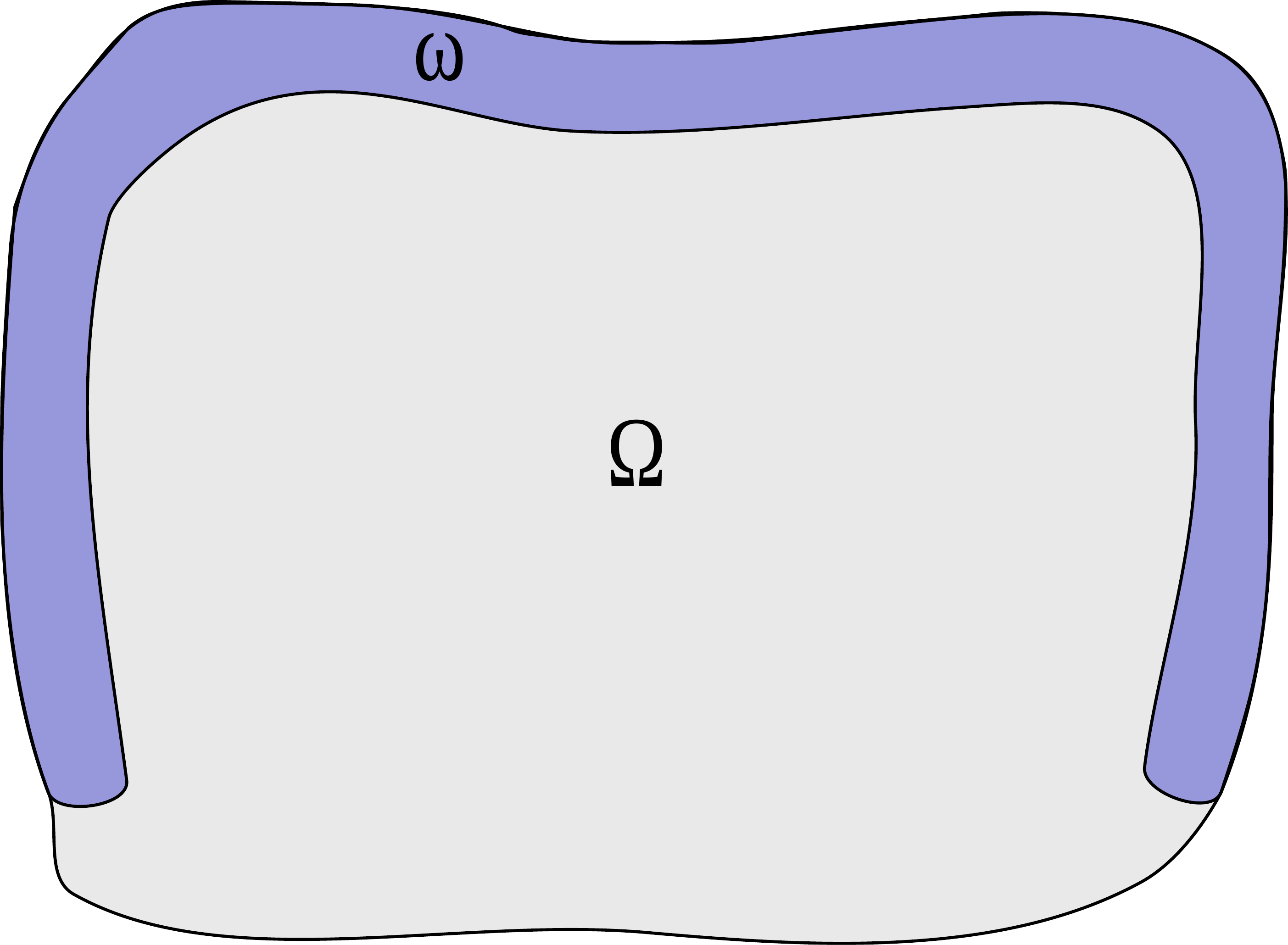}\\
				\caption{Controlling from the interior touching the boundary.}\label{interior_boundary}
			\end{center}
		\end{figure}
		
		In this case, as already pointed out in \cite{dehman2009analysis}, some difficulties in finding regular controls may arise. Indeed, as indicated both in \cite{dehman2009analysis} and in \cite{ervedoza2010systematic} a crucial property needs to be verified in order to have controls in $C^0([0,T];H^s(\omega))$, namely 
		\begin{equation}\label{control_op_reg}
		BB^*(D(A^*)^k)\subset D(A^k)
		\end{equation}
		for $k=0,\dots,s$, where we have used the notation of the proof of Theorem \ref{th_2}.
		
		Right now, for any $k\in \mathbb{N}$ we have
		\begin{equation*}\label{domain_generator_interior_-1}
		D(A^k)=\left\{\begin{pmatrix}
		\psi_1\\
		\psi_2\\
		\end{pmatrix}  \Bigg| \begin{matrix}
		\psi_1\in H^{k+1}(\Omega),&  \Delta^j\psi_1=0 \ \mbox{on} \ \partial\Omega,&  0\leq j \leq \lfloor{k/2}\rfloor\\
		\psi_2\in H^k(\Omega),&  \Delta^j\psi_2=0 \ \mbox{on} \ \partial\Omega,&  0\leq j\leq \lfloor{(k+1)/2}\rfloor-1\\
		\end{matrix}
		\right\},
		\end{equation*}
		while
		\begin{equation}\label{domain_generator_boundary_-1}
		D((A^*)^k)=\left\{\begin{pmatrix}
		\psi_1\\
		\psi_2\\
		\end{pmatrix}  \Bigg| \begin{matrix}
		\psi_1\in H^{k}(\Omega),& \Delta^j\psi_1=0 \ \mbox{on} \ \partial\Omega,\  0\leq j\leq\lfloor{(k-1)/2}\rfloor\\
		\psi_2\in H^{k-1}(\Omega), & \Delta^j\psi_2=0 \ \mbox{on} \  \partial\Omega,\ 0\leq j\leq\lfloor{k/2}\rfloor-1\\
		\end{matrix}
		\right\}.
		\end{equation}
		Furthermore,
		\begin{equation*}
		BB^*=\begin{pmatrix}
		0&0\\
		\chi^2&0\\
		\end{pmatrix}
		\end{equation*}
		Then, \eqref{control_op_reg} is verified if and only if for any $\psi\in H^s(\Omega)$ such that
		\begin{equation*}
		(\Delta)^j(\psi)=0,\quad 0\leq j\leq\lfloor{(s-1)/2}\rfloor,\quad\mbox{a.e. on} \ \partial\Omega
		\end{equation*}
		the following hold
		\begin{equation}\label{prop_to_be_ver}
		(\Delta)^j(\chi^2\psi)=0,\quad 0\leq j\leq\lfloor{(s-1)/2}\rfloor,\quad\mbox{a.e. on} \ \partial\Omega.
		\end{equation}
		Choosing $\chi$ so that all its normal derivatives vanish on $\partial \Omega$
		\begin{itemize}
			\item in case $s<5$, we are able to prove \eqref{control_op_reg}. Then, by adapting the techniques of \cite[Theorem 5.1]{ervedoza2010systematic}, we have that our system is \textit{Smoothly Controllable} (Definition \ref{contr_2_abstract}), with $s(n)=\lfloor{n/2}\rfloor+1$. This enables us to prove Theorem \ref{th_2} in space dimension $n<8$.
			\item in case $s\geq 5$, in \eqref{prop_to_be_ver} the biharmonic operator $(\Delta)^2$ enters into play. By computing it in normal coordinates on the boundary, some terms appear  involving the curvature and $\frac{\partial}{\partial \xi_k}{\chi}\frac{\partial}{\partial v}\psi$, where $(\xi_1,\dots,\xi_{n-1})$ are tangent coordinates, while $v$ is the normal coordinate. In general, these terms do not vanish, unless $\partial \Omega$ is flat. Then, for $n\geq 8$, we are unable to deduce a constrained controllability result in case the internal control is localized along a subregion of $\partial \Omega$.
		\end{itemize}

		\section{Boundary control: proof of Theorem \ref{th_1}, Theorem \ref{th_4} and Theorem \ref{th_5}}
		\label{sec:4}
		
		This section is devoted to boundary control and is organized as follows:
		\begin{itemize}
			\item Subsection \ref{subsec:4.1}: proof of Lemma \ref{lemma_smooth_controllability_wave_boundary} and Theorem \ref{th_1};
			\item Subsection \ref{subsec:4.2}: proof of Theorem \ref{th_4};
			\item Subsection \ref{subsec:4.3}: proof of Theorem \ref{th_5}.
		\end{itemize}
		
		\subsection{Proof of Theorem \ref{th_1}}
		\label{subsec:4.1}
		
		We prove Theorem \ref{th_1}.

		First of all, we explain how our boundary control system fits the abstract semigroup setting described in section \ref{sec:2}. The generator of the free dynamics is:
		\begin{equation}\label{generator_boundary_1}
		A=\begin{pmatrix}
		0&I\\
		-A_0&0\\
		\end{pmatrix},
		\hspace{0.5 cm}
		\begin{cases}
		H=L^2(\Omega)\times H^{-1}(\Omega)\\
		D(A)=H^1_0(\Omega)\times L^2(\Omega),
		\end{cases}
		\end{equation}
		where $A_0=-\Delta+cI:H^1_0(\Omega)\subset H^{-1}(\Omega)\longrightarrow H^{-1}(\Omega)$.
		The definition of the control operator is subtler than in the internal control case. Let $\Delta_0$ be the Dirichlet Laplacian. Then, the control operator
		\begin{equation*}
		B(v)=\begin{pmatrix}
		0\\
		-\Delta_{0} \tilde{z}\\
		\end{pmatrix},\quad\mbox{where}\ 
		\begin{cases}
		-\Delta\tilde{z}=0\hspace{0.6 cm} & \mbox{in} \hspace{0.10 cm}\Omega\\
		\tilde{z}=\chi v(\cdot,t) & \mbox{on}\hspace{0.10 cm} \partial\Omega.\\
		\end{cases}
		\end{equation*}
		defined from $L^2(\Gamma)$ to $H^{-\frac32}(\Omega)$. In this case, $B$ is unbounded but admissible (see  \cite{lions1988exact} or \cite[proposition 10.9.1 page 349]{OCO}).
		
		\begin{lemma}\label{lemma_smooth_controllability_wave_boundary}
			In the above framework, set $V=H^{s(n)-\frac12}(\Gamma)$ and $s=s(n)$, with $s(n)=\lfloor{n/2}\rfloor+1$. Suppose (GCC) holds for $(\Omega,\Gamma_0,T^*)$. Then, in any time $T_0>T^*$, the control system \eqref{wave_boundary_1} is smoothly controllable in time $T_0$.
		\end{lemma}One can prove the above Lemma, by employing \cite[Theorem 5.4]{ervedoza2010systematic}.

		\begin{proof}[Proof of Theorem \ref{th_1}]
			We prove our Theorem, by choosing the set of admissible controls:
			\begin{equation*}
			\mathscr{U}_{\mbox{\tiny{ad}}}=\left\{u\in L^2(\Gamma) \ | \ u\geq 0, \ \mbox{a.e.} \ {\Gamma}\right\}.
			\end{equation*}
			
			Hence,
			
			\begin{equation}\label{strict_constrants_1}
			\bigcup_{\sigma >0}\left\{u\in L^2(\Gamma) \ | \ u\geq \sigma, \ \mbox{a.e.} \ {\Gamma}\right\}\subset 	\mathscr{W}.
			\end{equation}
			
			Note that, in order to show \eqref{strict_constrants_1}, it is essential that the embedding\\
			$H^{s(n)-\frac12}(\Gamma)\hookrightarrow C^0(\overline{\Gamma})$ is continuous. This is guaranteed by the choice $s(n)=\lfloor{n/2}\rfloor+1$.

			By Lemma \ref{lemma_smooth_controllability_wave_boundary}, we conclude that smooth controllability holds. At this point, it suffices to apply Theorem \ref{th_6} to conclude.
			\qed\end{proof}

		\subsection{Proof of Theorem \ref{th_4}}
		\label{subsec:4.2}
		
		We prove now Theorem \ref{th_4}.
		\begin{proof}[Proof of Theorem \ref{th_4}]
			We have explained above how our control system \eqref{wave_boundary_1} fits the abstract framework presented in section \ref{sec:2}. Furthermore, by Lemma \ref{lemma_smooth_controllability_wave_boundary}, the system is \textit{Smoothly Controllable} with $s(n)=\lfloor{n/2}\rfloor+1$ and $V=H^{s(n)-\frac12}(\Gamma)$. Moreover, the set $\mbox{int}^V(\mathscr{U}_{\mbox{\tiny{ad}}}\cap V)$ is non empty, since all constants $\sigma>0$ belong to it.
			This is consequence of the continuity of $H^{s(n)-\frac12}(\Gamma)\hookrightarrow C^0(\overline{\Gamma})$, valid for $s(n)=\lfloor{n/2}\rfloor+1$. The result holds as a consequence of Theorem \ref{th_8}.
			\qed
		\end{proof}
		
		\subsection{State Constraints. Proof of Theorem \ref{th_5}}
		\label{subsec:4.3}
		
		We conclude this section proving Theorem \ref{th_5} about state constraints.
		The following result is needed.
		
		\begin{lemma}\label{lemma_boundary}
			Let $s\in \mathbb{N}^*$ and $T>T^*$. Take a steady state solution $\eta_0$ associated to the control $v^0\in H^{s-\frac12}(\Gamma)$. Then, there exists $v\in \cap_{j=0}^{s}C^j([0,T];H^{s-\frac12-j}(\Gamma))$ such that the unique solution $(\eta,\eta_t)$ to \eqref{wave_boundary_1} with initial datum $(\eta_0,0)$ and control $v$ is such that $(\eta(T,\cdot),\eta_t(T,\cdot))=(0,0)$. Furthermore,
			\begin{equation}\label{control_estimate_boundary_1}
			\sum_{j=0}^s\|v\|_{C^j([0,T];H^{s-\frac12-j}(\Gamma))}\leq C(T)\|v^0\|_{H^{s-\frac12}(\Gamma)},
			\end{equation}
			the constant $C$ being independent of $\eta_0$ and $v^0$. Finally, if $s=s(n)=\lfloor{n/2}\rfloor+1$, then the control $v\in C^0([0,T]\times \overline{\Gamma})$ and
			\begin{equation}\label{control_estimate_boundary_2}
			\|v\|_{C^0([0,T]\times\overline{\Gamma})}\leq C(T) \|v^0\|_{H^{s(n)-\frac12}(\Gamma)}.
			\end{equation}
		\end{lemma}
		
		The above Lemma can be proved by using the techniques of Lemma \ref{lemma_10}. We now prove our Theorem about \textit{state constraints}.
		
		\begin{proof}[of Theorem \ref{th_5}]

			\textit{Step 1} \ \textbf{Consequences of Lemma \ref{lemma_boundary}.}
			
			Let $T_0>T^*$, $T^*$ being the critical time given by the \textit{Geometric Control Condition}. By Lemma  \ref{lemma_boundary}, for any $\varepsilon >0$, there exists $\delta_{\varepsilon} >0$ such that for any pair of steady states $y_0$ and $y_1$ defined by regular controls $\overline{u}^i\in H^{s(n)-\frac12}(\Gamma)$, such that:
			\begin{equation}\label{th5_smallness_condition}
			\|\overline{u}^1-\overline{u}^0\|_{H^{s(n)-\frac12}(\Gamma)	}<\delta_{\varepsilon}
			\end{equation}
			we can find a control $u$ driving \eqref{state_equation_abstract} from $y_0$ to $y_1$ in time $T_0$ and verifying
			\begin{equation}\label{th5_linf_est_control_eps}
			\sum_{j=0}^{s(n)}\|u-\overline{u}^1\|_{C^j([0,T_0];H^{s(n)-\frac12 -j}(\Gamma))}<\varepsilon,
			\end{equation}
			where $\overline{u}^1$ is the control defining $y_1$. Moreover, if $(y,y_t)$ is the unique solution to \eqref{wave_boundary_1} with initial datum $(y_0,0)$ and control $u$, we have
			\begin{equation*}
			\|y-y_1\|_{C^0([0,T_0]\times \overline{\Omega})}\leq C\|y-y_1\|_{C^0([0,T_0];H^{s(n)}(\Omega))}
			\end{equation*}
			\begin{equation*}
			\leq C\sum_{j=0}^{s(n)}\|u-\overline{u}^1\|_{C^j([0,T_0];H^{s(n)-\frac12 -j}(\Gamma))}\leq C\varepsilon,
			\end{equation*}
			where we have used the boundedness of the inclusion $H^{s(n)}(\Omega)\hookrightarrow C^0(\overline{\Omega})$ and the continuous dependence of the data
			.
			\\
			\textit{Step 2} \ \textbf{Stepwise procedure and conclusion.}\\
			We consider the convex combination $\gamma(s)=(1-s)y_0+sy_1$. Then, let
			\begin{equation*}
			z_k=\gamma\left(\frac{k}{\overline{n}}\right),\quad k=0,\dots,\overline{n}
			\end{equation*}
			be a finite sequence of steady states defined by the control $\overline{u}_k=\frac{\overline{n}-k}{\overline{n}} \hspace{0.01 cm} \overline{u}^0+\frac{k}{\overline{n}} \hspace{0.01 cm} \overline{u}^1$. Let $\delta >0$. By taking $\overline{n}$ sufficiently large,
			\begin{equation}\label{th5_eq3}
			\|\overline{u}_{k}-\overline{u}_{k-1}\|_{H^{s(n)-\frac12}(\Gamma)}<\delta.
			\end{equation}
			By the above reasonings, choosing $\delta$ small enough, for any $1\leq k\leq \overline{n}$, we can find a control $u^k$ joining the steady states $z_{k-1}$ and $z_{k}$ in time $T_0$, with
			\begin{equation*}
			\|y^k-z_k\|_{C^0([0,T_0]\times \overline{\Omega})}\leq \sigma,
			\end{equation*}
			where $(y^k,(y^k)_t)$ is the solution to \eqref{wave_boundary_1} with initial datum $z_{k-1}$ and control $u^k$. Hence,
			\begin{equation}\label{th5_positivenss_contr_onestep_2}
			y^k=y^k-z_k+z_k\geq -\sigma+\sigma=0,\quad\mbox{on}\ (0,T_0)\times \Omega,
			\end{equation}
			where we have used the maximum principle for elliptic equations (see \cite{BFP}) to assert that $z^k\geq \sigma$ because $u_k\geq \sigma$.
			
			By taking the traces in \eqref{th5_positivenss_contr_onestep_2}, we have $u^k\geq 0$ for $1\leq k\leq \overline{n}$.
			
			In conclusion, the control
			$
			u:(0,\overline{n}T_0)\longrightarrow H^{s(n)-\frac12}(\Gamma)
			$
			defined as
			$u(t)=u_k(t-(k-1)T_0)$ for $t\in ((k-1)T_0,k T_0)$ is the required one. This finishes the proof.
			\qed\end{proof}
		
		\section{The one dimensional wave equation}
		\label{sec:5}
		
		We consider the one dimensional wave equation, controlled from the boundary
		\begin{equation}\label{wave_1_d_boundary_1}
		\begin{cases}
		y_{tt}-y_{xx}=0\hspace{0.6 cm} & (t,x)\in (0,T)\times (0,1)\\
		y(t,0)=u_0(t),\ y(t,1)=u_1(t)  & t\in(0,T)\\
		y(0,x)=y^0_0(x), \ y_t(0,x)=y^1_0(x).  & x\in (0,1)\\
		\end{cases}
		\end{equation}
		As in the general case, by transposition (see \cite{lions1988exact}), for any initial datum
		$(y_0^0,y_0^1)\in L^2(0,1)\times H^{-1}(0,1)$ and controls $u_i\in L^2(0,T)$, the above problem admits an unique solution $(y,y_t)\in C^0([0,T];L^2(0,1)\times H^{-1}(0,1))$.
		
		We show how Theorem \ref{th_4} reads in this one-dimensional setting, in the special case where both the initial trajectory $(\overline{y}_0,(\overline{y}_0)_t)$ and the final one $(\overline{y}_1,(\overline{y}_1)_t)$ are constant (independent of $x$) steady states.
		
		We determine explicitly a pair of \textit{nonnegative} controls steering \eqref{wave_1_d_boundary_1} from one positive constant to the other. The controlled solution remains \textit{nonnegative}.
		
		In this special case, we show further that
		\begin{itemize}
			\item the minimal controllability time is the same, regardless whether we impose the positivity constraint on the control or not;
			\item constrained controllability holds in the minimal time.
		\end{itemize}
		
		The minimal controllability time for \eqref{wave_1_d_boundary_1} is defined as follows.
		
		Let $(y_0^0,y_0^1)\in L^2(0,1)\times H^{-1}(0,1)$ be an initial datum and $(y_1^0,y_1^1)\in L^2(0,1)\times H^{-1}(0,1)$ be a final target. Then the minimal controllability time without constraints is defined as follows:
		\begin{equation}\label{def_min_time_unconstrained}
		T_{\mbox{\tiny{min}}} \overset{{\tiny \mbox{def}}}{=} \inf\left\{T>0 \ \big| \ \exists u_i\in L^{2}(0,T), \  (y(T,\cdot),y_t(T,\cdot))=(y_1^0,y_1^1)\right\}.
		\end{equation}
		Similarly, the minimal time under \textit{positivity} constraints on the \textit{control} is defined as:
		\begin{equation}\label{def_min_time_constrained_control}
		T_{\mbox{\tiny{min}}}^c \overset{{\tiny \mbox{def}}}{=} \inf\left\{T>0 \ \big| \ \exists u_i\in L^{2}(0,T)^{+}, \  (y(T,\cdot),y_t(T,\cdot))=(y_1^0,y_1^1)\right\}.
		\end{equation}
		Finally, we introduce the minimal time with constraints on the state and and the control:
		\begin{equation}\label{def_min_time_constrained_control&state}
		T_{\mbox{\tiny{min}}}^s \overset{{\tiny \mbox{def}}}{=} \inf\left\{T>0 \ \big| \ \exists u_i\in L^{2}(0,T)^{+}, \  (y(T,\cdot),y_t(T,\cdot))=(y_1^0,y_1^1), \ y\geq 0\right\}.
		\end{equation}
		
		\begin{figure}[htp]
			\begin{center}
				\includegraphics[width=6cm]{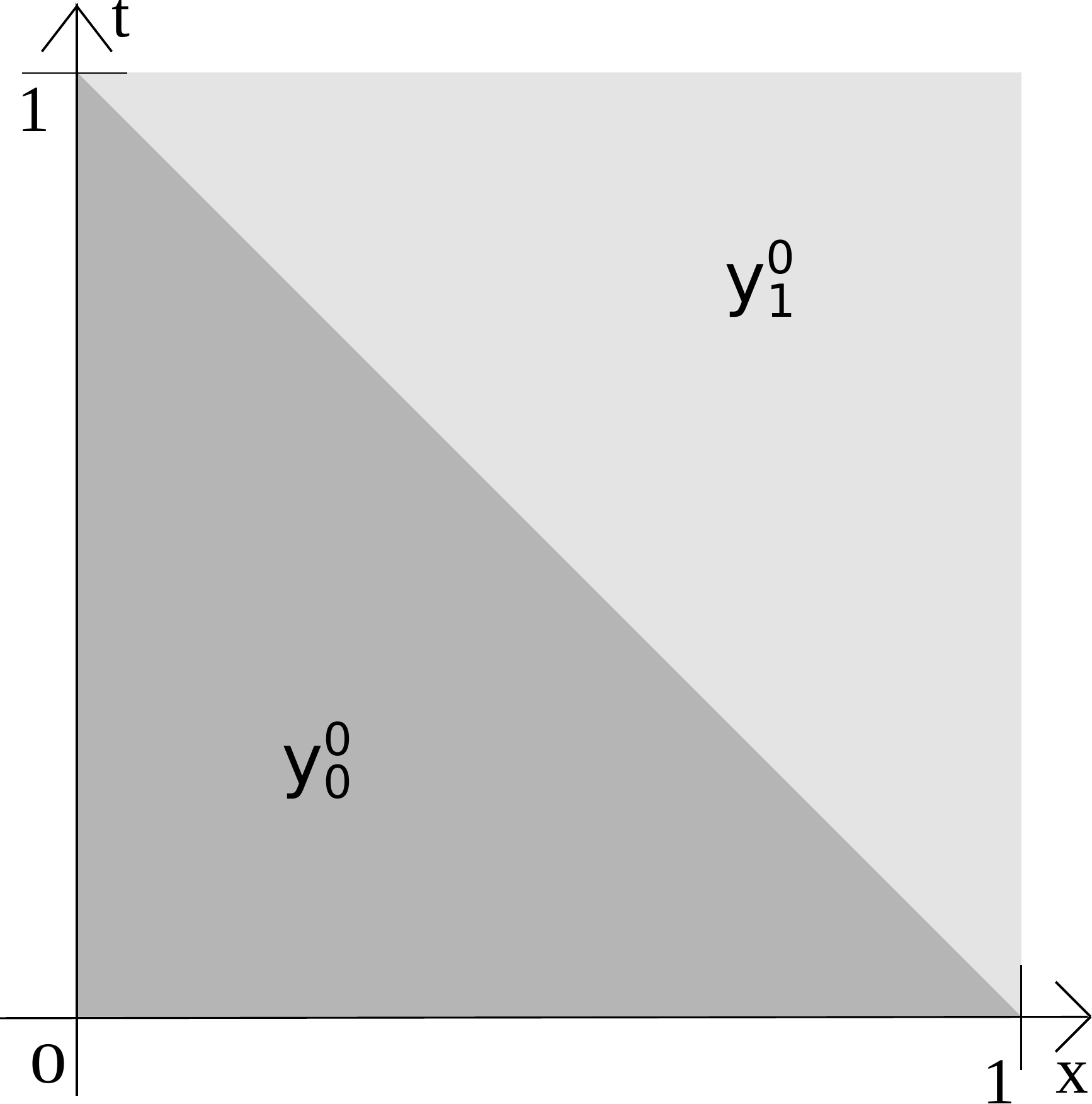}\\
				\caption{Level sets of the solution to \eqref{wave_1_d_boundary_1} with initial datum $(y_0^0,0)$ and controls $\hat{u}^i$. In the darker region the solution takes value $y_0^0$, while in the complement it coincides with $y_1^0$.}\label{contr minimal one d}
			\end{center}
		\end{figure}
		
		The problem of controllability of the one-dimensional  wave equation under bilateral constraints on the control has been studied in \cite{gugat2011optimal}. In the next Proposition, we concentrate on unilateral constraints and we compute explicitly the minimal time for the specific data considered.
		
		\begin{proposition}\label{prop_one_dimensional_1}
			Let $(y_0^0,0)$ be the initial datum and $(y_1^0,0)$ be the final target, with $y_0^0\in \mathbb{R}^+$ and $y_1^0\in\mathbb{R}^+$. Then,
			\begin{enumerate}
				\item for any time $T>1$, there exists two \textit{nonnegative} controls
				\begin{equation}\label{one_dimensional_controls}u_0(t)=
				\begin{cases}
				y_0^0 \quad & t\in [0,1)\\
				(y_1^0-y_0^0)\frac{t-1}{T-1}+y_0^0 \quad & t\in (1,T]\\
				\end{cases}
				\end{equation}
				\begin{equation}\label{one_dimensional_controls_2}
				u_1(t)=
				\begin{cases}
				(y_1^0-y_0^0)\frac{t}{T-1}+y_0^0 \quad & t\in [0,T-1)\\
				y_1^0 \quad & t\in [T-1,T]\\
				\end{cases}
				\end{equation}
				driving \eqref{wave_1_d_boundary_1} from $(y_0^0,0)$ to $(y_1^0,0)$ in time $T$. Moreover, the corresponding solution remains \textit{nonnegative}, i.e.
				\begin{equation*}
				y(t,x)\geq 0,\quad\forall (t,x)\in [0,T]\times [0,1].
				\end{equation*}
				\item $T_{\mbox{\tiny{min}}}^s=T_{\mbox{\tiny{min}}}^c=T_{\mbox{\tiny{min}}}=1$;
				\item the \textit{nonnegative} controls $\hat{u}_0\equiv y_0^0$ and $\hat{u}_1\equiv y_1^0$ in $L^2(0,1)$ steers \eqref{wave_1_d_boundary_1} from $(y_0^0,0)$ to $(y_1^0,0)$ in the minimal time. Furthermore, the corresponding solution $y\geq 0$ a.e. in $(0,1)\times (0,1)$;
				\item the controls in the minimal time are \textit{not unique}. In particular, for any $\lambda\in [0,1]$, $\hat{u}^0_{\lambda}=(1-\lambda)y_0^0+\lambda y_1^0$ and $\hat{u}^1_{\lambda}=(1-\lambda)y_1^0+\lambda y_0^0$ drives \eqref{wave_1_d_boundary_1} from $(y_0^0,0)$ to $(y_1^0,0)$ in the minimal time.
			\end{enumerate}
		\end{proposition}
		\begin{proof}
			We proceed in several steps.
			
			\textit{Step 1.} \ \textbf{Proof of the constrained controllability in time $T>1$.}\\
			By D'Alembert's formula, the solution $(y,y_t)$ to \eqref{wave_1_d_boundary_1} with initial datum $(y_0^0,0)$ and controls $u_i$ defined in \eqref{one_dimensional_controls} and \eqref{one_dimensional_controls_2}, reads as
			\begin{equation*}
			y(t,x)=f(x+t),\quad (t,x)\in [0,T]\times [0,1],
			\end{equation*}
			where
			\begin{equation*}f(\xi)=
			\begin{cases}
			y_0^0 \quad & \xi\in [0,1)\\
			(y_1^0-y_0^0)\frac{\xi-1}{T-1}+y_0^0 \quad & \xi\in [1,T)\\
			y_1^0. \quad & \xi\in [T,T+1].\\
			\end{cases}
			\end{equation*}
			This finishes the proof of $(1.)$.\\
			\textit{Step 2} \ \textbf{Computation of the minimal time.}\\
			In any time $T>1$, controllability under state and control constraints holds. Then,  $T_{\mbox{\tiny{min}}}\leq T_{\mbox{\tiny{min}}}^c\leq T_{\mbox{\tiny{min}}}^s\leq 1$.
			
			It remains to prove that $T_{\mbox{\tiny{min}}}\geq 1$. This can be obtained by adapting the techniques of \cite[Proposition 4.1]{loheac2016norm}.\\
			\textit{Step 3} \ \textbf{Controllability in the minimal time.}\\
			One can check (see figure \ref{contr minimal one d}) that the unique solution to \eqref{wave_1_d_boundary_1} with initial datum $(y_0^0,0)$ and controls $\hat{u}^i$ is
			\begin{equation}\label{one_dimensional_controls_3}y(t,x)=
			\begin{cases}
			y_0^0 \quad & t+x<1\\
			y_1^0 \quad & t+x>1\\
			\end{cases}
			\end{equation}
			This concludes the argument.
			\qed
		\end{proof}

		\section{Conclusions and open problems}
		\label{sec:6}
		
		In this paper we have analyzed the controllability of the wave equation under \textit{positivity} constraints on the control and on the state.
		\begin{enumerate}
			\item In the general case (without assuming that the energy defines a norm), we have shown how to steer the wave equation from one steady state to another in time large, provided that both steady states are defined by positive controls, away from zero;
			\item in case the energy defines a norm, we have generalized the above result to data lying on trajectories. Furthermore, the controls defining the trajectory are supposed to be only nonnegative, thus allowing us to take as target $(y_1^0,y_1^1)=(0,0)$.
		\end{enumerate}
		
		We present now some open problems, which as long as we know, have not been treated in the literature so far.
		
		\begin{itemize}
			\item Further analysis of controllability of the wave under \textit{state} constraints. As pointed out in \cite{HCC} and \cite{pighin2017controllability}, in the case of parabolic equations a state constrained result follows from a control constrained one by means of the comparison principle. For the wave equation, such principle does not hold. We have proved Theorem \ref{th_5},  using a ``stair-case argument'' but further analysis is required.
			\item On the minimal time for constrained controllability. Further analysis of the minimal constrained controllability time is required. In particular,  it would be interesting to compare the minimal constrained controllability time and the unconstrained one for any choice of initial and final data. As we have seen in Proposition \ref{prop_one_dimensional_1}, they coincide for constant steady data in one space dimension.
			\item In the present paper, we have determined nonnegative controls by employing results of controllability of smooth data by smooth controls. This imposes a restriction to our analysis: the action of the control is localized by smooth cut-off functions. In particular, when controlling \eqref{wave_internal_1} from an interior subset touching the boundary, we encounter the issues discussed in subsection \ref{subsec:3.3} and already pointed out in \cite{dehman2009analysis} and \cite{ervedoza2010systematic}.
			
			Then, it would be worth to be able to build nonnegative controls without using smooth controllability.
			\item Derive the Optimality System (OS) for the controllability of the wave by nonnegative controls.
			\item Extend our results to the semilinear setting, by employing the analysis carried out in \cite[Theorem 1.3]{cannarsa1999well}, \cite{cannarsa2002one}, \cite{cannarsa1999controllability} and  \cite{zuazua1993exact}.
			\item Extend the results to more general classes  of potentials  $c$. For instance, one could assume $c$ to be bounded, instead of $C^{\infty}$ smooth.
		\end{itemize}

		\section*{Appendix}
		\addcontentsline{toc}{section}{Appendix}

		\subsection*{Regularity results}
		In what follows, $H$ is a real Hilbert space and $A:D(A)\subset H\longrightarrow H$ is a generator of a $C^0$-semigroup.
		
		\begin{lemma}\label{lemma_semigroup_1}
			Let $k\in\mathbb{N}$. Take $y\in C^k([0,T];H)\cap H^{k+1}((0,T);H_{-1})$ solution to the homogeneous equation:
			\begin{equation}\label{eq_semigrpoup_1}
			\frac{d}{dt}y=Ay, \quad  t\in (0,T).
			\end{equation}
			Then, $y\in \cap_{j=0}^k C^j([0,T];D(A^{k-j}))$ and
			\begin{equation*}
			\sum_{j=0}^k\|y\|_{C^j([0,T];D(A^{k-j}))}\leq C(k)\|y\|_{C^k([0,T];H)},
			\end{equation*}
			the constant $C(k)$ depending only on $k$.
		\end{lemma}
		
		The proof of the above Lemma can be done by using the equation \eqref{eq_semigrpoup_1} (see \cite{BFP}).
		
		We prove now Lemma \ref{lemma_semigroup_4}.

		\begin{proof}[Proof of Lemma \ref{lemma_semigroup_4}]
			\textit{Step 1} \ \textbf{Time regularity}\\
			By induction on $j=0,\dots,k$, we prove that $y\in C^{j}([0,T];H)$ and
			\begin{equation*}
			\|y\|_{C^j([0,T];H)}\leq C\|f\|_{H^j((0,T);H)}.
			\end{equation*}
			For $j=0$, the validity of the assertion is a consequence of classical semigroup theory (e.g. \cite[Proposition 4.2.5]{OCO} with control space $U=H$ and control operator $B=Id_{H}$). Assume now that the result hold up to $j-1$. Then, let $w$ solution to
			\begin{equation}\label{lemma_semigroup_4_eq3}
			\begin{cases}
			\frac{d}{dt}w=Aw+f^{\prime}& t\in (0,T)\\
			w(0)=0.\\
			\end{cases}
			\end{equation}
			By induction assumption, $w\in C^{j-1}([0,T];H)$ and the corresponding estimate holds. Then, $\tilde{y}(t)=\int_0^t w(\sigma) d\sigma\in C^j([0,T];H)$ and
			\begin{equation*}
			\|\tilde{y}\|_{C^j([0,T];H)}\leq C\|f\|_{H^j((0,T);H)}.
			\end{equation*}
			Then, it remains to show that $y=\tilde{y}$. Now, for any $t\in [0,T]$
			\begin{equation*}
			\tilde{y}(t)-\tilde{y}(0)=\int_0^t[w(\sigma)-w(0)]d\sigma =\int_0^t\int_0^{\sigma}[Aw(\xi)+f^{\prime}(\xi)]d\xi d\sigma
			\end{equation*}
			\begin{equation*}
			=\int_0^t[A\tilde{y}(\sigma)+f(\sigma)]d\sigma.
			\end{equation*}
			By uniqueness of solution to \eqref{state_equation_2}, we have $y=\tilde{y}$. This finishes the first step.\\
			\textit{Step 2} \ \textbf{Conclusion}\\
			We start observing that $y$ solves
			\begin{equation*}
			y_t=Ay,\qquad t\in (\tau,T).
			\end{equation*}
			Then, by classical semigroup arguments (see \cite[Chapter 7]{BFP}), we conclude.
			\qed
		\end{proof}

		\subsection*{Proof of Lemma \ref{lemma_10}}
		
		We give the proof of Lemma \ref{lemma_10}.
		
		\begin{proof}[Proof of Lemma \ref{lemma_10}]
			Let $v$ be given by \eqref{control_noreg_smooth}. The proof is made of two steps.
			\\
			\textit{Step 1} \ \textbf{Show that $y(1;\eta_0,v)\in D(A^s)$, with $s$ given by \eqref{s and V}}
			\\
			We apply Lemma \ref{lemma_semigroup_4} with $y=y(\cdot;\eta_0,\rho\overline{v}^0)-\rho \eta_0$ and $f=\rho^{\prime}\eta_0$, getting 
			\begin{equation*}
			y(1;\eta_0,\rho\overline{v}^0)-\rho\eta_0\in D(A^s).
			\end{equation*}
			Since $\rho\eta_0=0$ over $(\delta,1)$, for some $\delta\in (0,1)$, we have that 
			$$y(1;\eta_0,\rho\overline{v}^0)\in D(A^s).$$
			\textit{Step 2} \ \textbf{Conclusion}\\
			Since $y(1;\eta_0,\rho\overline{v}^0)\in D(A^s)$, we are in position to apply the smooth controllability (see Definition \ref{contr_2_abstract}) and determine $w\in L^{\infty}((1,T_0+1);V)$ steering the solution to \eqref{state_equation_abstract} from $y(1;\eta_0,v)$ at time $t=1$ to $0$ at time $t=T_0+1$.\\
			Hence, the desired control $v$ reads as \eqref{control_noreg_smooth}.
			
			Finally, by similar reasonings the estimate \eqref{control_norm_estimate_1} follows. This ends the proof of this Lemma.
			\qed
		\end{proof}
		
		\subsection*{Proof of Lemma \ref{lemma_stair_case_6}}
		
		We prove now Lemma \ref{lemma_stair_case_6}.
		
		\begin{proof}[Proof of Lemma \ref{lemma_stair_case_6}]
			We split the proof in two steps.\\
			\textit{Step 1} \ \textbf{Proof of the inequality $\|\mathbb{T}_t\|_{\mathscr{L}(D(A^s),D(A^s))}\leq 1$ with $t\in \mathbb{R}^+$}\\
			Recall that
			\begin{equation*}
			\|x\|_{D(A^s)}^2=\sum_{j=0}^s\|A^jx\|_H^2\quad\forall \ x\in D(A^s).
			\end{equation*}
			Now, for any $x\in D(A^s)$ and $t\in\mathbb{R}^+$, we have
			\begin{equation*}
			\|A^j\mathbb{T}_tx\|_H=\|\mathbb{T}_tA^jx\|_H\leq \|A^jx\|_H\quad\forall \ j=0,\dots,s.
			\end{equation*}
			This yields $\|\mathbb{T}_t\|_{\mathscr{L}(D(A^s),D(A^s))}\leq 1$ for any $t\in\mathbb{R}^+$.\\
			\textit{Step 2} \ \textbf{Conclusion.}\\
			Let $C>0$ be given by \eqref{contr_2_abstract}. Take
			\begin{equation}\label{lemma_stair_case_6_1}
			N>\frac{C\|\eta_0\|_{D(A^s)}}{\varepsilon}.
			\end{equation}
			Arbitrarily fix $k\in\left\{0,\dots,N-1\right\}$. Consider the following equation
			\begin{equation}\label{lemma_stair_case_6_2}
			\begin{cases}
			\frac{d}{dt}y(t)=Ay(t)+B\chi_{(kT_0,(k+1)T_0)}(t)u_k(t)& t\in \mathbb{R}^+\\
			y(0)=\frac{1}{N}\eta_0,\\
			\end{cases}
			\end{equation}
			where $\chi_{(kT_0,(k+1)T_0)}$ is the characteristic function of the set $(kT_0,(k+1)T_0)$ and $u_k\in L^2(\mathbb{R}^+,V)$. From step 1 and \eqref{lemma_stair_case_6_1}, we have that
			\begin{equation}\label{lemma_stair_case_6_3}
			\|y(kT_0;(1/N)\eta_0,0)\|_{D(A^s)}\leq (1/N)\|\eta_0\|_{D(A^s)}\leq \varepsilon.
			\end{equation}
			Then, we apply smooth controllability (given by ($H_1^{\prime}$)) to find some control $\hat{u}_k\in L^{\infty}(\mathbb{R}^+;V)$ so that the solution to \eqref{lemma_stair_case_6_2} with control $u_k=\hat{u}_k$ satisfies
			\begin{equation}\label{lemma_stair_case_6_4}
			y((k+1)T_0;(1/T_0)\eta_0,\chi_{(kT_0,(k+1)T_0)}\hat{u}_k)=0\quad\mbox{and}\quad\|\hat{u}_k\|_{L^{\infty}((kT_0,(k+1)T_0);V)}\leq \varepsilon.
			\end{equation}
			Now, we define:
			\begin{equation}\label{lemma_stair_case_6_5}
			v(t)=\sum_{k=0}^{N-1}\chi_{(kT_0,(k+1)T_0)}(t)u_k(t)\quad t\in\mathbb{R}^+.
			\end{equation}
			Then, from \eqref{lemma_stair_case_6_4} and \eqref{lemma_stair_case_6_5}, we know
			\begin{equation*}
			y(NT_0;\eta_0,v)=0\quad\mbox{and}\quad \|v\|_{L^{\infty}((0,NT_0);V)}\leq \varepsilon.
			\end{equation*}
			This leads to the conclusion where $\overline{T}=NT_0$.
			\qed\end{proof}

	\bibliography{my_references}
	\bibliographystyle{siam}
	
	
\end{document}